\documentclass[11pt]{amsart}

\usepackage[utf8]{inputenc}
\usepackage[T1]{fontenc}
\usepackage{amsmath}
\usepackage{amsthm}
\usepackage{amssymb}
\usepackage{mathrsfs}
\usepackage{graphicx}
\usepackage[nice]{nicefrac}
\usepackage{textcomp}
\usepackage{enumerate}
\usepackage{hyperref}
\usepackage{amscd}
\usepackage[all,cmtip]{xy}

\newcounter{foo}
\newcounter{foo2}

\makeatletter
\newtheorem*{rep@theorem}{\rep@title}
\newcommand{\newreptheorem}[2]{%
\newenvironment{rep#1}[1]{%
 \def\rep@title{#2 \ref{##1}}%
 \begin{rep@theorem}}%
 {\end{rep@theorem}}}
\makeatother

\newtheorem{thm}[foo]{Theorem}
\newreptheorem{thm}{Theorem}
\newtheorem{lem}[foo]{Lemma}
\newtheorem{cor}[foo]{Corollary}
\newtheorem{prop}[foo]{Proposition}

\theoremstyle{remark}

\newtheorem{rem}[foo2]{Remark}

\newcommand{\RR}{\mathbb{R}}

\newcommand{\NN}{\mathbb{N}}

\newcommand{\LL}[1]{L^{#1}}

\newcommand{\LPN}[2]{\left\|#2\right\|_{\LL{#1}}}
\newcommand{\HN}[2]{\left\|#2\right\|_{H^{#1}}}
\newcommand{\ip}[1]{\langle#1\rangle}
\newcommand{\abs}[1]{\left|#1\right|}
\newcommand{\norm}[1]{\left\|#1\right\|}
\newcommand{\cinfty}{C^{\infty}}
\newcommand{\testf}{C^{\infty}_{0}}

\DeclareMathOperator{\sech}{sech}

\DeclareMathOperator{\KdV}{KdV}
\DeclareMathOperator{\mKdV}{mKdV}
\DeclareMathOperator{\miura}{M}

\title{The Korteweg--de Vries equation at $H^{-1}$ regularity}      
\author{Tristan Buckmaster}     
\address{Mathematisches Institut der Universität  Leipzig and Max-Planck-Insitut f\"ur Mathematik in den Naturwissenschaften}
\email{tristan.buckmaster@gmail.com}
\author{Herbert Koch} 
\address{Mathematisches Institut der Universität Bonn}
\email{koch@math.uni-bonn.de} 
 \thanks{This work was supported  by the DFG through the Hausdorff center for Mathematics, Bonn}
\date{}        

\begin{document}
\maketitle
\begin{abstract} In this paper we will prove the existence of weak
  solutions to the Korteweg--de Vries initial value problem on the
  real line with $H^{-1}$ initial data; moreover, we will study the
  problem of orbital and asymptotic $H^{s}$ stability of solitons for integers $s\geq -1$; finally, we will also
  prove new a priori $H^{-1}$ bound for solutions to the Korteweg--de
  Vries equation. The paper will utilise the Miura transformation to
  link the Korteweg--de Vries equation to the modified Korteweg--de
  Vries equation. \end{abstract}

\section{Introduction and statement of result}
Consider the initial value problem (IVP) of the \emph{Korteweg-de Vries} (KdV) equation:
\begin{equation}\label{KdV}
\left\{\begin{array}{l}
u_t + u_{xxx} - 6 u u_x  =0\\
u(0,x) = u_0(x)
\end{array}\right.,
\end{equation}
for $x\in \RR$ and rough initial data $u_0$ in the Sobolev space $H^s$.

It is well known that the KdV equation exhibits special travelling
wave solutions, known as solitons -- indeed such solutions provided
much of the historical impetus to study the equation.  Explicitly, up
to a spatial translation, these solutions may be written in the form
\begin{equation}
u:=R_c(x-ct),
\end{equation}  
where $c>0$ and
\begin{equation}
R_c:=-\frac{c}{2}\sech^2\left(\frac{\sqrt{c}x}{2}\right).
\end{equation}

Let us summarise the well-posedness theory and stability results. 
The initial value for KdV is known to be globally
well-posed\footnote{See \cite{MR2233925} for a discussion on the
  subtleties in the definition of well-posedness.} for $s\geq
-\frac{3}{4}$, (see \cite{MR2531556}, \cite{MR1969209}, 
\cite{MR2233925} and \cite{MR2501679}).  The problem is known to be ill-posed for $s<
-\frac{3}{4}$ in the sense that the flow map cannot be uniformly
continuous  \cite{MR1813239}.  One may hope for Hadamard
well-posedness for $s\ge -1$, (cf. \cite{MR2353092}, \cite{2011arXiv1112.5177L} and
\cite{MR2267286}). Using the inverse scattering transform, Kappeler and Topalov proved that the flow map extends 
continuously to $H^{-1} $ in the periodic case, 
which provides motivation to address  the supposedly simpler question of well-posedness in $H^{-1}$ 
on the real line.  On the other hand Molinet \cite{MOLINET} has shown that 
no well-posedness can possibly hold below $s=-1$:
the solution map $u_0
\to u(t)$ does not extend to a continuous map from $H^s$, for $s < -1$, to
distributions.

Orbital stability of the soliton in the energy space $H^1$ follows
from Weinstein's convexity argument \cite{MR820338}, this argument even holds for other
sub-critical gKdV equations. Weinstein's argument is at the basis of a
considerable amount of work since then, with one direction culminating 
in the seminal work of Martel and Merle to some version of asymptotic
stability, again in the energy space \cite{MR1826966}.  Merle and Vega
proved orbital stability and asymptotic stability of the soliton
manifold in $L^2$ using the Miura map \cite{MR1949297} in a similar
fashion to our approach.  Their approach to the stability of the kink
however is closer to the arguments of Martel and Merle for generalised
KdV.

We now present our principal results.

\begin{thm}[New $H^{-1}$ a priori estimate for KdV]\label{aprioriTHM}
Suppose $s \geq -\frac34$ and $u\in C([0,\infty);H^s(\RR))$ is a solution to \eqref{KdV}, then
\begin{equation}\label{apriori}
\HN{-1}{u(t,\cdot)} \lesssim \footnotemark  \HN{-1}{u_0}+ \HN{-1}{u_0}^3   ~~\text{ for }~~ t\in[0,\infty).
\end{equation}\footnotetext{Throughout this article we will adopt the notation $a\lesssim b$ to signify $a\leq Cb$, where $C$ is an insignificant constant.}
\end{thm}
\begin{rem}
Applying scaling, the dependence on the $H^{-1}$ norm of the initial data in \eqref{apriori} can be made more explicit, i.e. if $\lambda$ is such that
\[ 0 < \lambda  \le    \HN {-1}{ u_0}^{-2},   \]
then we have 
\begin{equation*}
  \HN{-1}{u(t,\lambda \cdot)} \lesssim  \HN{-1}{u_0(\lambda \cdot ) }  ~~\text{ for }~~ t\in[0,\infty).
\end{equation*}
\end{rem}

\begin{thm}[Orbital stability of KdV solitons]\label{KdVstab} 
There exists an $\varepsilon >0$ such that 
if $u\in C([0,\infty);H^s(\RR)\cap H^{-3/4}(\RR))$ is a solution to (\ref{KdV}), for some integer $s\geq-1$, satisfying $\HN{-1}{R_c-u_0}<\varepsilon c^{1/4}$ for some $c>0$, then there is a continuous function $y:[0,\infty)\to\RR$ such that 
\[ \HN{s}{ u - R_{ c}(x-y(t))} \leq \gamma_{s}(c,\HN{s} {R_c-u_0})\] 
for any $t\geq 0$, where $\gamma_{s}:(0,\infty)\times [0,\infty)$ is a continuous function, 
polynomial in the second variable, which  satisfies $\gamma(\cdot,0)=0$.
\end{thm} 

\begin{rem}
If we rescale $c$ to $4$, we obtain a more precise result. The smallness assumption becomes 
\[ \HN {-1}{ 4c^{-1} u_0(2c^{-1/2}x) - R_4 } \le \varepsilon, \]
which is weaker and more natural than the assumption of the theorem. 
\end{rem}

\begin{thm}[Asymptotic stability of KdV solitons]\label{KdVAsymStab}
  Given real $\gamma>0$ and integer $s\geq -1$, there exists an
  $\varepsilon_{\gamma} >0$ such that if $u\in
  C([0,\infty);H^s(\RR)\cap H^{-3/4}(\RR))$, is a solution to
  (\ref{KdV}), satisfying 
\[ \HN{-1}{R_c-u_0}<\varepsilon_{\gamma}  c^{1/4} \]
for $c>0$, then there is a continuous function
  $y:[0,\infty)\to\RR$ and $\tilde c>0$ such that
\begin{equation*}  
 \lim_{t \to \infty} \norm{ u - R_{\tilde c}(x-y(t)) }_{H^{s}( (\gamma t , \infty) ) } =  0 
\end{equation*} 
for any $t\geq 0$.  Moreover we have the bound $\abs{c-\tilde  c}\lesssim  c^{\frac34} \HN{-1}{R_c-u_0}   $.
\end{thm} 
The decay follows from an explicit quantitative estimate in
Proposition \ref{stableftright2} for $H^{-1}$, and similar estimates for higher 
norms in   Corollary \ref{mKdVAsymStab} .  The estimates we obtain are
sufficiently strong to obtain existence of weak solutions by a
standard approximation and compactness argument.

\begin{thm}[Existence of global $H^{-1}$ weak solutions to KdV IVP]\label{mainTheorem}
For any $u_0\in H^{-1}$, there exists a weak solution $u$ to
(\ref{KdV}) satisfying
\begin{align}
u\in C_{\omega}([0,\infty); H^{-1}(\RR)), \footnotemark\\
u\in L^2([0,T]\times[-R,R]) ~~\text{for any}~~ R,T<\infty,\\
u(t,\cdot) \to u_0 \qquad \text{ in } H^{-1} ~~\text{as}~~t\downarrow 0.
\end{align}\footnotetext{Here $C_{\omega}([0,\infty); H^{-1}(\RR))$ denotes the space of weakly continuous functions from $\RR$ to $H^{-1}$.}
Furthermore $u$ satisfies the bounds given in Theorem \ref{aprioriTHM}.
\end{thm}

A closely related problem to the initial value problem of the
Korteweg-de Vries equation is that of the \emph{modified} Korteweg--de
Vries (mKdV) equation:
\begin{equation}\label{mKdV}
\left\{\begin{array}{l}
u_t + u_{xxx} - 6 u^2 u_x=0\\
u(0,x) = u_0(x)
\end{array}\right.,
\end{equation}
for $x\in \RR$ and initial data $u_0$.  

An explicit family of solutions of the mKdV equation, called \emph{kink}
solutions, can be written up to translations as
\[Q_\lambda(t,x):=\lambda\tanh\left(\lambda x+2\lambda^3 t\right),\]
for any $\lambda>0$.

The mKdV problem and the KdV may be connected via a differential transformation known as the \emph{Miura map}:
\begin{equation}
u\mapsto u_x+u^2;
\end{equation}
which sends solutions of (\ref{mKdV}) to solutions of (\ref{KdV}).  To see this property formally, set
\begin{align*}
\KdV(u)&=u_t + u_{xxx} - 6 u u_x,&\\
\mKdV(u)&=u_t + u_{xxx} - 6 u^2 u_x,&~\text{and}\\
\miura(u)&=u_x+u^2.&
\end{align*}
One can then easily check that
\begin{equation}\label{miuraid}
\KdV(\miura(u))=(\mKdV(u))_x+2u\cdot \mKdV(u),
\end{equation}
from which it follows that $\KdV(\miura(u))=0$ whenever $\mKdV(u)=0$.  Additionally, note the mKdV equation satisfies the reflection symmetry: if $u$ is a solution to \eqref{mKdV}, then $-u$ is also a solution. Hence if we define
\begin{equation*}
\miura^*(u):=\miura(-u)=-u_x+u^2,
\end{equation*}
then $\miura^*(u)$ maps solutions to the mKdV equation to solutions to the KdV equation.

The Korteweg--de Vries equation is invariant under the \emph{Galilean
  transformation}:
\begin{equation}
u(t,x)\mapsto u(t,x-ht) - \frac{h}{6},
\end{equation}
for $h\in \RR$, i.e. if $u$ is a solution to the KdV equation, then
its image under the above transformation is also a solution, which is
easily verified.

The Korteweg--de Vries equation also satisfies the following scaling symmetry:
\begin{equation}\label{scaling}
u(t,x)\mapsto\frac{1}{\lambda^2}u(\frac{t}{\lambda^3},\frac{x}{\lambda}), 
\end{equation}
for $\lambda>0$ and $\dot H^{-\frac32}$ is the critical space. 

In Section \ref{invmiurasec} we will show how to use the Miura map
combined with the Galilean symmetry to relate mKdV solutions near a
kink solutions to either KdV solutions near $0$, or to KdV solutions
near a soliton.  This will afford us the freedom to choose the most
convenient setting in order to prove the stated results.  The $H^{-1}$
a priori estimate of Theorem \ref{aprioriTHM} will then follow as a
consequence of the $L^2$ stability of the kink (Theorem \ref{stab}).
The orbital (Theorem \ref{KdVstab}) and asymptotic stability (Theorem
\ref{KdVAsymStab}) of the soliton in the $H^{-1}$ norm will follow
from the corresponding statement for the mKdV kink in $L^2$ (Theorem
\ref{stab} and Theorem \ref{mKdVAsymStabL2}).  Higher conserved
energies imply stability of the trivial solution in $H^s$ for
nonnegative integers $s$, and Kato's local smoothing argument along a
moving frame implies asymptotic stability of the trivial solution to
the right. We use the Miura map to derive 
 orbital and asymptotic stability 
of the soliton for KdV, as well as orbital and asymptotic  stability of the kink for mKdV in higher norms, requiring  smallness of the deviation only in $H^{-1}$, 
see Corollary \ref{highernorm} and the proof of Corollary
\ref{mKdVAsymStab}.

The Miura map has been used in a simpler setting by Kappeler et al.\
\cite{MR2189502}. Their results are limited by the fact that the Miura
map is not invertible. Our additional ingredient is the shift of the
initial data using the Galilean invariance.  To the best of our knowledge
the corresponding results on the Miura map are new, and we believe
them to be appealing and of independent interest.

Of course this is intimately related to the integrable structure of
KdV and mKdV, and also the Lax-Pair is clearly in the
background. Nevertheless we do not explicitly use the integrable
machinery, and the use of elementary key elements of the theory of
integrable systems in combination with a PDE oriented approach seems
to be new and promising.

It is worthwhile to point out that unlike the corresponding asymptotic
stability results for generalised KdV, the scale $\tilde c$ is
independent of time. This holds since the scale of the kink is related
to its size at infinity, and this does not change by adding $L^2$
perturbations.

\section{Inverting the Miura map}\label{invmiurasec}

Kappeler et al.\ showed in the paper \cite{MR2189502} that if the
initial data $u_0\in H^{-1}$ is contained in the image of $L^2$ under
the Miura map restricted to $L^2$, then there exists a global weak
solution to the IVP \eqref{KdV}. The proof consists of constructing a
weak solution to mKdV corresponding to initial data in the preimage
under the Miura map of the original initial data, and then
transforming the solution back, under the Miura map, to a solution to
KdV.  The following proposition is one of the key tools used by
Kappeler et al.\ to characterise the range of the Miura map.

\begin{prop}\cite{MR2189502}\label{spectrum}
Let $ u_0 \in H^{-1}_{\text{loc}}$. The following three statements are equivalent. 
\begin{enumerate}
\item The Schrödinger operator $H_{u_0}:=-\partial_{xx}+ u_0$ is positive
  semi-definite.
\item There exists a strictly positive function $\phi $ with $-\phi_{xx} + u_0 \phi=0$.
\item  $u_0\in H^{-1}_{\text{loc}}$ is in the range of the Miura map on $L^2_{\text{loc}}$.  
\end{enumerate}
\end{prop} 

In order to remove the restrictions on the initial data imposed in \cite{MR2189502}, we will employ the use of
the Galilean transformation in order to transform KdV into the range of the Miura map.    This allows us to
link rough $H^{-1}$ KdV initial data to mKdV initial data.  The
corresponding mKdV initial data will be in the form of a sum of an
$L^2$ function and a $\tanh$ kink.  The authors would like to note
that the original idea to use such an argument  was somewhat motivated by
the papers \cite{MR1949297} and \cite{TZV} -- related to the $L^2$
stability of soliton solutions to the KdV equation and KP-II equation
respectively.

Appendix \ref{appenSchr} contains results for  Schrödinger operators with $H^{-1}$ potentials, which we will use below.  

Our aim now is to
construct an ``inverse'' of the Miura map for Galilean transformed
initial data.  The Galilean transformation essentially adds a
constant to the potential of the Schrödinger operator corresponding to
the initial data. We easily achieve positive definiteness by adding a
large enough constant; the caveat being that the initial data will no
longer remain in $H^{-1}$.

Given initial data $u_0\in H^{-1}$, applying the Galilean symmetry to
$u_0$, with $h$ set to $-6\lambda^2$, for some $\lambda>0$ and $t=0$, yields the function $u_0+\lambda^2$.  Now
consider the problem of finding a function $r\in L^2$ that is in the
preimage of $u_0+\lambda^2$ under the Miura transformation.  Observe that
$\miura( \lambda\tanh (\lambda \cdot))=\lambda^2$; it then seems natural to consider
the problem
\begin{equation}\label{prob1}
(r+\lambda \tanh \lambda x)_x+(r +\lambda\tanh \lambda x)^2=u_0+\lambda^2.
\end{equation}

For the problem of stability of solitons, we assume we are given some initial data $u_0\in H^{s}$, where $s\geq -\frac{3}{4}$.  Applying the Galilean
transform with $h$ as above, and   noting that $\miura^*(\lambda
\tanh(\lambda\cdot))=\lambda^2-2\lambda^2\sech^2(\lambda\cdot)$,
we are led to consider the problem
\begin{equation}\label{prob2}
-(r+\lambda \tanh (\lambda \cdot))_x+(r +\lambda\tanh (\lambda \cdot))^2= u_0+\lambda^2.
\end{equation}

We now state sufficient and necessary conditions for the problems \eqref{prob1}) and \eqref{prob2} to have a solution.

\begin{prop}\label{groundstates} Let $\lambda>0$. The  
 ground state energy  of $H_{u_0}$, $u_0 \in H^{-1}$ 
is $-\lambda^2$ if and only if there exists $r \in L^2-\lambda\tanh(\lambda \cdot)$ such that 
\[   \miura(r) = u_0+\lambda^2. \]
The spectrum of $H_{u_0}$ is contained in $(-\lambda^2,\infty)$ if and only 
if there  exists $r \in L^2 + \lambda\tanh(\lambda \cdot)$ with 
\[ \miura(r) =u_0+\lambda^2. \]
\end{prop} 

\begin{proof} 
Let $\phi$ be the ground state. Observe then $r= \frac{d}{dx} \ln \phi$ satisfies the Ricatti equation (see Appendix \ref{appenSchr}) 
\[ r_x+r^2 = u_0 +\lambda^2. \]
Then as a consequence of Lemma \ref{asymp} we have either
\[ r-\lambda  \in L^2(0,\infty) \quad\text{or}\quad r+\lambda \in L^2(0,\infty). \]
Note however the property that
\[  e^{\int_0^{x} r } \in L^2 \]
enforces $r+\lambda  \in L^2(0,\infty)$. Similarly, we obtain $r-\lambda \in L^2(-\infty,0)$ and thus 
\[ r + \lambda \tanh(\lambda x) \in L^2. \]
Hence $u_0+\lambda^2$ is in the range of the Miura map on $ - \lambda \tanh(\lambda x)+ L^2$ if the ground state energy is $-\lambda^2$. 

Now assume that 
\[  \lambda^2+  u_0 = r_x + r^2 \quad \text{and}\quad   r+\lambda \tanh(\lambda x) \in L^2.   \]
Then $ \phi = e^{\int_0^{x} r}$ is a strictly positive function in $H^1$ satisfying
\[(H_{u_0}+\lambda^2)\phi=0,\]
i.e. $\phi$ is the ground state with ground state energy $-\lambda^2$.

Now we turn to the case when the spectrum is contained in $(-\lambda^2,\infty)$. Since $H_{u_0+\lambda^2}$ is positive semi-definite, by Proposition \ref{spectrum}, there exists strictly positive $\phi\in L^2_{\text{loc}}$ satisfying
\begin{equation}\label{abc}
(H_{u_0} + \lambda^2)\phi=0.
\end{equation} 
Note that $\phi\notin L^2$, otherwise $\phi$ would be the ground state.  Since $\frac{d}{dx} \ln \phi$ solves the Ricatti equation, it follows by Lemma \ref{asymp} that either $\phi$ grows exponentially as $x\to \infty$ or as $x\to -\infty$. 

We aim to construct a solution $\tilde \phi$ to \eqref{abc} satisfying
\begin{equation}\label{expbds}
\tilde \phi(x) \to  \infty  \text{ as } x \to \pm \infty. 
\end{equation} 

Suppose $\phi$ is not such a solution; then without loss of generality we can assume 
\[ \phi(x) \to \infty \qquad \text{ as } x \to \infty, \]
and
\[ \phi(x) \to 0 \qquad \text{ as } x \to -\infty. \]

 We obtain using Lemma \ref{asymp}  that 
$\frac{d}{dx} \ln \phi - \lambda \in L^2$.   It it is then not difficult to show that
\begin{equation}
\tilde \phi(x)=C\phi(x)+ \phi(x)\int_0^x \phi^{-2}(s) ds
\end{equation}
for large $C>0$, is a solution to \eqref{abc}, satisfying the growth conditions \eqref{expbds}. 

We now define
   $r= \frac{d}{dx} \ln  \tilde\phi$.  It  satisfies 
the Ricatti equation; moreover, by Lemma \ref{asymp},  
\[ r-\lambda \tanh(\lambda x) \in L^2. \] 
Thus $u_0+\lambda^2$ is in the range of the Miura map restricted to $L^2+\lambda \tanh(\lambda x)$ if the
spectrum of $u_0$ is contained in $(-\lambda^2, \infty)$.

Now suppose that $u_0+\lambda^2 = r_x + r^2$ for $r - \lambda \tanh (\lambda x) \in L^2$; hence $\phi = e^{\int_0^x r }$
satisfies the equation
\[ -\phi'' + u_0 \phi = -\lambda^2 \phi, \]
and 
\[ \phi(x) \to \infty \text{ as } x \to \pm \infty. \]
Observe that Proposition \ref{spectrum} implies that the Schrödinger
operator $H_{u_0}$ has spectrum contained in $[-\lambda^2, \infty)$. We
  want to show that $-\lambda^2$ is not an eigenvalue. If it were,
  then there would be a non-negative, strictly positive $L^2$ ground
  state $\psi$. This is not possible since $\phi/\psi$ cannot attain a
  minimum (see Lemma \ref{maximum}). Therefore the spectrum is
  contained in $(-\lambda^2, \infty)$.
\end{proof} 

We now turn to the problem of 
relating the two sides of \eqref{prob1} and \eqref{prob2} 
by analytic diffeomorphisms.  We begin with a technical statement.

\begin{lem}\label{techlem} 
The multiplication map $(u, v)\to uv$ can be extended from the
bilinear map $C_0^\infty\times C_0^\infty\to C_0^\infty$ to continuous
bilinear maps
\begin{align*}
&L^2\times L^2 \to L^1\subset H^{s},&\text{ for any }
  s<-\frac12\\ &L^2\times H^{s^{\prime}} \to H^{s},& ~-\frac12 <s\le
  0,~s^{\prime}>\frac12+s\\ &H^{s_1}\times H^{s_2} \to H^{s_1},&\text{
    for any } s_2>\frac12,~ 0\le s_1, \le s_2.
\end{align*}
\end{lem}
\begin{proof}
The first two statements may be proved using Sobolev embedding
inequalities and their corresponding dual inequalities.  The last case
is a particular case of Theorem 1, of Section 4.6.1 of
\cite{MR1419319}, alternatively it may be proved by interpolating the
second statement with the well known algebraic property of Sobolev
spaces $H^s$ for $s>\frac{1}{2}$.
\end{proof}

For $s\geq -1$, let $F_{\lambda}: H^{s+1} \to H^{s}\times \mathbb{R} $
and $F^*: H^{s+1} \times (0,\infty) \to H^{s}$ to be the maps:
\begin{align*}
  F_{\lambda}(r)&= \left( r^2+2r\lambda\tanh(\lambda x)+r_x , \int r \sech^2(\lambda x) ~dx \right)  ,
\\  F^*(r,\lambda)&=r^2+2r\lambda\tanh(\lambda x)-r_x-2\lambda^2\sech^2(\lambda\cdot) .
\end{align*}
It then follows from  Lemma \ref{techlem} , that the above maps define
quadratic (neglecting $\lambda$ for $F^*$ here), and hence analytic maps from $H^{s+1} \to H^{s}\times \RR$ and $H^{s+1} \times (0,\infty) \to H^{s}$, respectively.
Analyticity in $\lambda$ (and joined analyticity) follows from the obvious 
holomorphic extension of  $\lambda$ into the complex plane.

The equations
\begin{equation*}
r^2+2r\tanh x+r_x=u_0,
\end{equation*}
and
\begin{equation*}
r^2+2r\lambda\tanh (\lambda\cdot)-r_x-2\lambda^2\sech^2(\lambda\cdot)= u_0;
\end{equation*}
relating functions in range and image come from the expansion of the 
left hand sides of  \eqref{prob1} and \eqref{prob2} respectively.

Now let $L_{\lambda,r}$ denote the first component of the Fréchet
derivative at $r$, and similarly let $L_{\lambda,r}^*$ denote the
Fréchet derivative of $F^*$ with respect to the first component at
$(r,\lambda)$, i.e.
\begin{equation*}
L_{\lambda,r}v := 2\left(\lambda\tanh (\lambda \cdot)+r\right)v+v_x,
\end{equation*}
and $L_{\lambda,r}^*$ is its formal adjoint: 
\[ L_{\lambda,r}^*  v :=  2\left(\lambda \tanh(\lambda \cdot) +r\right)v -  v_x . \]

\begin{lem} \label{opL}
For any $s \in \mathbb{R}$ and $r\in H^{s}\cap L^2$, the abstract operator
$L_{\lambda,r}$ and its formal adjoint operator $L_{\lambda,r}^*$
define bounded operators from $H^{s+1}(\mathbb{R})$ to
$H^s(\mathbb{R})$, which we denote by $L_{\lambda,r}$ and
$L^*_{\lambda,r}$, respectively, suppressing $s$ from the notation.

Both $L_{\lambda,r}$ and $L_{\lambda,r}^*$ are Fredholm operators of index $1$ and $-1$ respectively. 
Moreover, setting $\phi_r=\sech^2( \lambda \cdot)e^{-2\int_0^{\cdot}r~dy}$, the operator $L_{\lambda,r}$ is surjective, with null space spanned by $\phi_r$; and the formal adjoint $L_{\lambda,r}^*$ is injective with closed range and cokernel spanned by $\phi_r$. 

\end{lem} 
\begin{proof}

  It follows from Lemma \ref{techlem} that the operators
  $L_{\lambda,r}$ and $L_{\lambda,r}^*$ define continuous linear
  operators from $H^{s+1}$ to $H^s$. A simple calculation shows that
  \[ L_{\lambda,r} \sech^2(\lambda x)\exp\left(-2\int_0^{x}r~dy\right) = 0. \] Since $L_{\lambda,r}\phi=0$
  is a scalar ordinary differential equation, every solution is a
  multiple of $\sech^2(\lambda x)\exp\left(-2\int_0^{x}r~dy\right)$ and the null space is one
  dimensional.  Similarly, one can easily check that $L_{\lambda,r}^*$
  is injective (since solutions to the homogeneous equation are
  multiples of $\cosh^2(\lambda x)\exp\left(2\int_0^{x}r~dy\right)$) and $\sech^2(\lambda x)\exp\left(-2\int_0^{x}r~dy\right)$ spans
  the cokernel of $L_{\lambda,r}^*$.

To complete the proof we need to show $L_{\lambda,r}$ is surjective, and $L_{\lambda,r}^*$ is injective with closed range.
  By scaling, it suffices to show the case when $\lambda=1$.  For reasons of brevity we will use the shorthand $L_{r}:=L_{1,r}$, and $L:=L_{1,0}$.  
  
   We start by defining an integral operator from $L^2$ to $H^{1}$
  which is a right inverse of $L_r$.  We begin with  the simpler case when $r=0$.

Let $\eta\in\testf([-2,2])$ be a non-negative function such that $\eta\equiv 1$ on $[-1,1]$ and consider the operator  $T$ defined by
\begin{align*}
T(g)=& e^{-2\int_{0}^x \tanh s~ds} \int_{-\infty}^x e^{2\int_{0}^y \tanh s^{\prime}~ds^{\prime}}\eta(y)g(y) ~dy+\\& e^{-2\int_{0}^x \tanh s ~ds} \int_0^x e^{2\int_{0}^y \tanh s^{\prime}~ds^{\prime}} \left(1-\eta(y)\right)g(y) ~dy.
\end{align*}

Note that $T$ is a well defined operator for functions in $\testf$.  
It can then be easily checked that $Tg$  satisfies $LTg=g$.  Now let $K(x,y)$ be the kernel of $T$:
\begin{equation}
    K(x,y) \equiv
    \left\{ \begin{array}{ll}
        \eta(y)\cosh^2 y\sech^2 x & y < x \leq 0\\
        \eta(y)\cosh^2 y\sech^2 x & y < 0 < x\\
        -(1-\eta(y))\cosh^2 y\sech^2 x & x \leq y \leq 0\\
        \cosh^2 y\sech^2 x & 0 \leq y < x\\
        0 & \textit{otherwise}
    \end{array}\right. 
\end{equation}

Now consider the case for general $r$: formally we have 
\begin{equation*}
T_r g := \exp\left(-2\int_0^{\cdot}r\right)T\exp\left(2\int_0^{\cdot}r\right)g,
\end{equation*}
satisfies $L_rT_r g=g$; furthermore the kernel of $T_r$ is given by
\begin{equation*}
K_r(x,y)=K(x,y)\exp\left(-2\int_0^x r + 2\int_0^y r\right).
\end{equation*}

We now claim that 
\[ \HN{1}{ T_rg } \lesssim \LPN 2 g . \]

Observe that 
\begin{equation}
K_r(x,y)\lesssim e^{-2\abs{y-x}+\sqrt{\abs{y-x}}\LPN{2}{r}}\lesssim e^{-\abs{y-x}+\LPN{2}{r}^2},
\end{equation} and hence
\begin{equation}\label{L2OpBd}
\LPN{2}{T_r g}\lesssim \LPN{2}{e^{-\abs{\cdot}}*g}\lesssim \LPN{2}{g}.
\end{equation}
The equality  
\begin{equation}\label{RegIter}
\partial_x T_rg +2\left( \tanh(x) +r\right)T_r g = g, 
\end{equation}  
implies 
\begin{align*}
\LPN{2}{ \partial_x T_rg } &\le 2 \LPN{2}{ T_rg }+2\LPN{2}{rT_rg}+ \LPN{2}{ g }\\
 &\lesssim \LPN{2}{g}+\LPN{2}{r}\LPN{\infty}{T_rg}\\
&\lesssim \LPN{2}{g}+\LPN{2}{r}\LPN{2}{\partial_x T_rg}^{1/2}\LPN{2}{T_rg}^{1/2},
\end{align*} 
where we used the $L^2$ estimate \eqref{L2OpBd}, Hölder's inequality and Gagliardo-Nirenberg's inequality.  Finally applying Young's inequality and \eqref{L2OpBd} again, we obtain 
\begin{equation}\label{thirds}
\LPN{2}{ \partial_x T_rg} \lesssim \left(1+\LPN{2}{r}^2\right)\LPN{2}{g}.
\end{equation} 

Hence if $r\in L^2$, then $T_r$ extends to a bounded operator from $L^2$ to $H^1$, satisfying $ L_r T_r g= g$, thus $L_r:H^1\rightarrow L^2$ is surjective.

By duality, for every $r\in L^2$, the adjoint operator $L^*_r: L^2 \rightarrow H^1$ is injective with closed range, or equivalently
\begin{equation}\label{zuiop}
\LPN 2 f \lesssim \HN {-1}{L^*_r f}   
\end{equation}
for all $f\in L^2$.

We will now show that given any $f\in H^{s+1}$ and $h\in H^s\cap L^2$, if 
\[g:=2\left(\tanh(\cdot) +h\right)f \pm  f_x,\]
then we have the following inequality
\begin{equation}\label{regiterstate}
\HN {s+1} f \leq C\left(\LPN 2 f+ \HN s g\right), 
\end{equation}
for some constant $C$ depending on $\HN s h + \LPN 2 h$.

First note the trivial estimate
\begin{align}
\label{seconds}\begin{split}
\HN{s+1}{f}&\lesssim \LPN 2 f+\HN{s}{\partial_x f}\\
&\lesssim \LPN 2 f+\HN{s}{g}+\HN{s}{f}+\HN{s}{h f}.
\end{split}
\end{align}

Consider the case for $-1\le s<-\frac12$: it follows from Lemma \ref{techlem} and \eqref{seconds} that
\begin{align}
\HN{s+1}{f}&\lesssim \LPN 2 f+\HN{s}{g}+\left(1+\LPN{2}{h}\right)\LPN{2}{f}.\label{firsts}
\end{align}

Now consider the case when $s\ge-\frac12$: again from Lemma \ref{techlem} and \eqref{seconds} we have 
\[\HN{s+1}{f}\lesssim \HN{s}{f}+\HN{s}{g}+\left(1+\HN s h+ \LPN 2 h\right)\HN{s+3/4}{f}.\]
Using \eqref{firsts} and applying the above estimate recursively we obtain \eqref{regiterstate}.

Combining \eqref{zuiop} with \eqref{regiterstate} ($h=r$), it follows
that for all $s\geq -1$, $f\in H^{s+1}$ and $r\in L^2\cap H^s$
\[
\HN {s+1}{f} \lesssim \HN{s}{L^*_r f},
\]
i.e. $L^*_r:H^{s+1}\rightarrow H^{s}$ is injective with closed range.

By duality it follows that if $r\in L^2$ then $L_r:L^2\rightarrow
H^{-1}$ is surjective; moreover, as a consequence of
\eqref{regiterstate} ($h=r$) we obtain $L_r:H^{s+1}\to H^{s}$ is
surjective for all $s \in \mathbb{R}$ and $r\in L^2\cap H^s$, and
again the full statement for $L^*_r$ follows.
\end{proof}

Let $U^s_{>\kappa}\subset H^s$ denote the set of all functions in $H^s$ whose
associated Schrödinger operator has spectrum contained in
$(\kappa,\infty)$.  Similarly, define $U^s_{<\kappa}$ to be the subset
of $H^s$ of all functions $f$ whose associated Schrödinger operator
has spectrum $\omega(T_f)$ such that $\omega \setminus
(\kappa,\infty)\neq\emptyset$.

\begin{thm}\label{inverse} 
  For any $s\geq -1$ the map $F_{\lambda}:H^{s+1}\rightarrow H^s\times\RR$ is an analytic diffeomorphism onto its range. Moreover, for any $f\in U^s_{>-\lambda^2}$, there exists $\rho\in \RR$ such that $(f,\rho)$ is contained in the range of $F_{\lambda}:H^{s+1}\rightarrow H^s\times\RR$.
  
  For any $s\geq -1$ the map $F^*:H^{s+1} \times (0,\infty)\rightarrow H^s$ is an analytic diffeomorphisms onto  $U^s_{<0}$.
\end{thm} 
\begin{proof}
First we show the two maps are local analytic diffeomorphisms.  
  
Note the second component of $DF_{\lambda}|_r$ is simply the map $f\mapsto \ip{f,\sech^2(\lambda \cdot)}$; hence from Lemma \ref{opL} we obtain $DF_{\lambda}|_r$ is
  invertible -- here we use the fact \[\ip{\phi_r,\sech^2(\lambda \cdot)}>0,\] where $\phi_r$ is defined as in Lemma \ref{opL}.  Hence by the inverse function theorem, $F_{\lambda}$ is a local analytic  diffeomorphism.
  
  Let $G:H^{-1}\rightarrow \RR$ be the map from potentials to the ground state energy of their corresponding Schrödinger operator.  By Proposition \ref{groundstates} we have that $G(F^*(f,\lambda))=-\lambda^2$, from which it follows that the derivative of $F^*$ with respect to the
  second component is not contained in the range of the derivative with respect to the first component.  Then from Lemma \ref{opL} and the implicit function theorem, it follows that $F^*$ is a local analytic  diffeomorphism. 

We now prove the injectivity of the two maps.
Suppose $r_i$, $i\in\{1,2\}$ satisfy $F_{\lambda}(r_1) = F_{\lambda}(r_2)$. Then, with $w= r_2-r_1$ we have
\[ w_x + (2\lambda \tanh(\lambda x) + r_2+r_1) w = 0 , \qquad \int \sech^2(\lambda x) w(x) dx =0. \]
The same argument as for invertibility implies that $w=0$, hence $r_1=r_2$ 
and $F_{\lambda}$ is injective. 

We now show $F^*$ is injective. First of all, if $F^*(r_1,\lambda_1)
=F^*(r_2,\lambda_2)$, then by Lemma \ref{groundstates} we necessarily have $\lambda_1=\lambda_2$. Next,  with $w=
r_2-r_1$
\[ w_x - (2\lambda \tanh(\lambda x) + r_1+r_2) w = 0. \]
 The only solution in $L^2$ to this equation is $w=0$. This implies injectivity.

Now we make the following observation: if $F_{\lambda}(r) = (g,\rho)$
for $r \in L^2$ and $g \in H^s$, $s >-1$ then we also have $r\in
H^{s+1}$; similarly, if $F^*(r,\lambda) = g$ for $r \in L^2$ and $g
\in H^s$, $s >-1$ then we also have $r\in H^{s+1}$. This follows by
iteratively applying \eqref{regiterstate}, with $h=r/2$.

Thus as a consequence of the Proposition \ref{groundstates}, together with the above observation, we obtain that the range of the projection  of $F_{\lambda}(r):H^{s+1}\rightarrow H^s\times \RR$ onto its first component is precisely $U^s_{>-\lambda^2}$, and the range of $F^*:H^{s+1} \times (0,\infty)\rightarrow H^s$ is $U^s_{<0}$.
\end{proof}

\begin{rem}\label{wvest}
Note that working out the details of the $H^{s+1}$ estimates in the proof above, one may show that if 
\[v(t,x):= w(t,x-y_0)^2+2w(t,x-y_0)\tanh(x-y_1)+w_x(t,x-y_0),\]
for some $y_0,y_1\in \RR$ then for every integer $s\ge -1$ there exists a $N>0$ such that
\begin{equation*}
\HN{s+1}{ w } \lesssim (1+\LPN 2 w + \HN s v )^N \left(\HN s v+\LPN 2 w\right).
\end{equation*} 
This estimate will be used later in Section \ref{seckink}.
\end{rem}

\section{The modified Korteweg-de Vries equation close to a kink}\label{seckink}

In Section \ref{invmiurasec}, we mapped $H^{s}$ KdV initial data to
mKdV initial data -- the mKdV initial data being in the form of a sum
of a $H^{s+1}$ function, and a kink of the form
$\lambda\tanh(\lambda \cdot)$.  In this section we will study the
corresponding mKdV problem.

First note by scaling, we may restrict to the case $\lambda=1$.  Recall that $Q_1(t,x)\equiv Q(t,x)\equiv \tanh(x+2t)$
is an  explicit kink solution to \eqref{mKdV}.  We
now consider solutions to mKdV equation:
\begin{equation}\label{uKINKmKdV}
\left\{\begin{array}{l}
u_t + u_{xxx} - 6u^2 u_x=0\\
u(0,x) = v_0(x)+\tanh(x)
\end{array}\right.,
\end{equation}
for initial data $v_0\in H^s$, $s\geq 0$, such that $u-Q\in H^s$.  Equivalently, writing $u=v+Q$, we have:
\begin{equation}\label{KINKmKdV}
\left\{\begin{array}{l}
v_t + v_{xxx} - 2 \partial_x (3Q^2 v+ 3Q v^2+v^3 )=0\\
v(0,x) = v_0(x)
\end{array}\right..
\end{equation}

In order to construct global solutions to \eqref{uKINKmKdV}, we will
need to prove a number of energy estimates.  In our discussions below,
we will use a number of formal calculations, which are not difficult
to justify rigorously.

\begin{lem}
Let $p$ be any $\cinfty(\RR^2)$ function with uniformly bounded derivatives and assume $v\in C([0,\infty); H^{1}(\RR))$ to be a solution to \eqref{KINKmKdV}.  Then
\begin{align}\label{nifty}
\begin{split}
\frac{d}{dt}  \left[ \int p v^2~dx\right]=    &\int \Big[ p_t v^2 -3 p_x v_x^2+ p_{xxx}v^2 \\
                                           &-6p_x Q^2 v^2-8p_x Q v^3-3p_xv^4 \\
                                           &+12p Q Q_{x} v^2+4pQ_x v^3 \Big]~dx.
\end{split}
\end{align} 
\end{lem}
\begin{proof}
Note that by \eqref{KINKmKdV}, $v$ also has some time regularity. Thus, equation \eqref{nifty} follows by employing \eqref{KINKmKdV} and applying a series of integrations by parts.
\end{proof}

We are now in the position to prove global bounds on the $L^2$ norm of smooth solutions to \eqref{KINKmKdV}, as well as a ``Kato smoothing'' type estimate.

\begin{lem}\label{localbd}
Suppose $v\in C([0,\infty); H^{1}(\RR))$ is a solution to (\ref{KINKmKdV}); then for any $t\in [0,\infty)$, we have 
\begin{equation}\label{badL2bd}
\LPN{2}{v(t,\cdot)} \lesssim \LPN{2}{v_0}+t^{1/2}.
\end{equation}
Moreover for any $T>0$
\begin{equation}\label{KATOSMOOTH}
\int_{0}^T\int_{-\infty}^{\infty} Q_x(t,x)v_x(t,x)^2~dx~dt\leq C(T,\LPN{2}{v(0,\cdot)}).
\end{equation}
\end{lem}
\begin{proof}
From (\ref{nifty}), with $p\equiv 1$, we have
\begin{equation}\label{mass1}
\frac {d}{dt} \left[ \int  v^2~dx\right]=    \int 12 Q Q_{x} v^2+4Q_x v^3 ~dx.
\end{equation} 

Using (\ref{nifty}) again, but with $p\equiv Q$, yields
\begin{align}\label{lastest}
\begin{split}
\frac{d}{dt} \left[ \int Q v^2~dx\right]=&    \int \Big[ Q_t v^2 -3 Q_x v_x^2+ Q_{xxx}v^2 \\
                                       &    -6Q^2 Q_x v^2-8Q Q_x v^3-3Q_xv^4 \\
                                        &   +12Q^2 Q_{x} v^2+4QQ_x v^3 \Big]~dx.
\end{split}
\end{align} 

Observe that the terms $Q^2Q_x v^2$, $QQ_x v^2$, $Q_t v^2$ and $Q_{xxx} v^2$ are all bounded above by a multiple of $Q_x^{1/2}v^2$. Furthermore, we have
\begin{equation}\label{YOUNG1}
\begin{split} \int Q_x v^2~dx\leq & \LPN{2}{Q_x^{1/2}}\LPN{2}{Q_x^{1/2}v^2}\\ \leq& \frac{1}{\tau}\int Q_x v^4~dx 
 + \tau \LPN{2}{Q_x^{1/2}}^2,
\end{split}
\end{equation} 
for $\tau >0$ by  Young's inequality.
Note also $\abs{QQ_x v^3}\leq \abs{Q_x v^3}$ and
\begin{align*}
\int \abs{Q_x v^3}~dx &\leq \LPN{2}{Q_x^{1/2}v}\LPN{2}{Q_x^{1/2}v^2}\\
&\leq \frac{1}{\kappa}\int Q_x v^4~dx + \kappa \int Q_x v^2~dx,
\end{align*}
for any $\kappa>0$.  Applying (\ref{YOUNG1}), we find that for any $\kappa>0$, there exists a constant $C_{\kappa}>0$, such that
\begin{equation}\label{YOUNG2}
\int \abs{Q_x v^3}~dx \leq \frac{1}{\kappa}\int Q_x v^4~dx + C_{\kappa}.
\end{equation}
Combining equations (\ref{mass1}-\ref{YOUNG2}) we obtain
\begin{equation}\label{equality}   
\frac {d}{dt} \left[ \int \left[v^2+\frac{1}{10} Q v^2\right]~dx\right]\lesssim 1.
\end{equation}

Since $v^2\lesssim v^2+ \frac{1}{10} Q v^2$ , we can conclude that for any $t\ge 0$
\begin{equation*}
\LPN{2}{v(t,\cdot)} \lesssim \LPN{2}{v(0,\cdot)}+t^{1/2}.
\end{equation*}

The proof of \eqref{KATOSMOOTH} follows from \eqref{badL2bd} and the estimate \eqref{lastest}.  
\end{proof}

As a consequence of the above estimates, we obtain the following well-posedness result for initial mKdV data near a kink. 

\begin{thm}\label{strongsol}
Let $s\in \NN$ satisfy $s\geq 1$.  Then there exists a unique, global, strong solution to (\ref{KINKmKdV}), for any initial data $v_0\in H^{s}$.  Moreover, for any $T>0$, the solution map from $H^s$ to $C_t ([0,T];H^s(\RR))$ is continuous.
\end{thm}

The proof of this theorem follows essentially from the same arguments
as those given in \cite{MR1211741} and \cite{MR1949297} -- since the
$L^2$ estimate \eqref{badL2bd} is available. 

We now establish global a priori bounds on the deviation of a solution
$u$ to \eqref{uKINKmKdV} from a translated kink, i.e. we aim to
establish bounds on $w:=u-\tanh(x-y(t))$ for some continuous function
$y:\RR^+\rightarrow \RR$ yet to be determined. We start by providing some motivation for the full argument given later.

 From \eqref{uKINKmKdV} we obtain
\begin{equation}\label{KINKmKdV2}
 \begin{array}{rl}  
w_t + w_{xxx} - 2 \partial_x (3\tanh^2(x-y(t))^2 w+ 3 \tanh(x-y(t))  w^2+w^3 )\hspace{-4cm} & 
\\
& 
= 
(\dot y + 2) \sech^2(x-y(t)) 
\end{array}.
\end{equation}

To define the position $y(t)$, we impose an orthogonality condition 
\begin{equation} \label{pos2}   \langle w , \psi(x-y) \rangle = 0, \end{equation} 
where for the moment we choose $\psi(x)= e^x \sech^2(x)$  for reason which will become clear 
below -- later we will actually choose $\psi(x)=\eta(x)\sech^2(x)$, where $\eta$ is defined by \eqref{etadef}. If $w$ is sufficiently close to the kink then $y$ exists and is unique 
by an application of the implicit function theorem -- see Lemma \ref{position} 
below\footnote{In Lemma \ref{position} the weight $\psi(x)=\eta(x)\sech^2(x)$ is actually used; however, the proof may be easily adapted to the case $\psi(x)= e^x \sech^2(x)$.}.  It is not hard to work out an equation for $\dot y +2$ by formally 
differentiating the condition \eqref{pos2} 
 with respect to $t$.

It is instructive to first consider the linearised problem at the $Q(x,t)=\tanh(x+2t)$ kink,
 in a frame moving with the kink:
\begin{equation}\label{lineareq}
\tilde w_t(t,x) -4\tilde w_x(t,x)+ \tilde w_{xxx}+6 \partial_x(\sech^2(x)\tilde w(t,x))= \alpha(t) \sech^2(x),
\end{equation} 
where $\alpha(t)$ is chosen as indicated above so that orthogonality condition 
\begin{equation}\label{motivortho}
 \langle \tilde w , e^{x} \sech^2(x) \rangle = 0 
\end{equation}
 is preserved over time. To obtain a formula for $\alpha$, we differentiate the above orthogonality condition with respect to $t$.  Thus we obtain
\[  \langle  \tilde w, (-4 \partial_x +\partial_{xxx} +6 \sech^2 \partial_x )e^x \sech^2(x)  \rangle 
 + \alpha(t) \langle \sech^2(x) , e^x \sech^2(x)  \rangle = 0, \]
which expresses $\alpha$ as a linear function of $\tilde w$.

 From \eqref{lineareq} we obtain
\begin{equation}\label{motiveq}  \frac{d}{dt} \int e^x \tilde w^2 dx = -3B(e^{x/2} \tilde w)+2\alpha(t)\int \tilde w(t,x) e^{x} \sech^2(x)~dx,
\end{equation} 
where here $B$ is the quadratic form defined by
\begin{equation}\label{newquadform}B(f):=\int f_x^2 + \bigg(\frac54 -2\sech^2(x)  - 4 \sech^2(x)\tanh(x)\bigg) f(x)^2~dx.\end{equation}
Note that the second term on the right hand side 
of \eqref{motiveq} vanishes due to \eqref{motivortho}.

According to the bound \eqref{assumption} of Proposition
\ref{quadform}, taking into account the choice of $\psi$, we have
\begin{equation*} B(e^{x/2} w) \ge \frac{1}{10} \HN 1 { e^{x/2} w }^2.
 \end{equation*}
As a consequence $\int e^x \tilde w^2 dx $ decays monotonically, and the time 
derivative controls $B(e^{x/2} \tilde w)$.

We will pursue a non-linear variant of this simple strategy.  The
weight $e^{x}$ will be replaced by the following bounded and monotone
weight function
 \begin{equation}\label{etadef}
  \eta_{R,\delta} (x)=\eta (x) =
 \tanh\left(\frac{x-R}{2}\right)+1+\delta.
\end{equation}  
We will also define $y$ in terms of an orthogonality condition, similar to \eqref{motivortho}.

\begin{lem} \label{position} 
There exists an $\epsilon >0$ and a unique 
analytic  function $y$ on $B^{L^2}(\tanh(\cdot),\epsilon)$ such that 
\[  \langle   f(\cdot)-  \tanh(\cdot-y(f)) ,  \eta( \cdot- y(f)) \sech^2(\cdot-y(f)) \rangle = 0, 
\]
and $y(\tanh(\cdot)) = 0 $. 
\end{lem} 
\begin{proof} 
Consider the mapping 
\[ F:B^{L^2}(\tanh(\cdot),\epsilon)\times \RR\rightarrow\RR \]
-- here $B^{L^2}(\tanh(\cdot),\epsilon)$ denotes by an abuse of notation 
the set of sums of $L^2$ functions of norm $<\varepsilon$ and $\tanh$ --
 defined by
\[ F(f,y)=\ip{ f - \tanh(\cdot-y) , \eta(\cdot-y)\sech^2(\cdot-y)}. 
\]
Clearly $F(\tanh(\cdot),0)=0$. Differentiating with respect to $y$ at $f:=\tanh(\cdot)$ we obtain 

\begin{equation*} \left. \frac{d}{dy} F(\tanh(\cdot),y) \right|_{y:=0} 
=     
\langle \sech^2(\cdot-y) ,\eta(\cdot-y) \sech^2(\cdot-y)\rangle >0 . 
\end{equation*}
The implicit function theorem then yields the assertion. 
\end{proof} 

\begin{thm}\label{stab} 
There exists a $\delta >0$
such that if $u$ the  solution to \eqref{uKINKmKdV} of Theorem \ref{strongsol} with initial data  satisfying $u(0,\cdot) - \tanh(\cdot) \in H^1(\mathbb{R})$  and $\LPN 2 {u(0,\cdot)-\tanh(\cdot)}<\delta$, then there is a continuous function $y:[0,\infty)\to\RR$ such that 
\begin{equation}
\LPN{2}{ u(t,.)  - \tanh(x-y(t)) } \lesssim \LPN 2 {u(0,\cdot)-\tanh(\cdot) }.\label{apriorimkdv}
\end{equation}
Moreover, writing $w:=u-\tanh(\cdot-y(t))$ we have the estimates
\begin{equation}\label{virial}
\int_0^\infty \norm{\eta_x(\cdot-y(t))^{1/2} w}_{H^1}^2~dt\lesssim \LPN{2}{u(0,\cdot)-\tanh(\cdot)}^2,
\end{equation}
and
\begin{align}\label{dotyest3}
\begin{split}
\abs{\dot{y}+2}\lesssim &\LPN 2 {\eta_x(\cdot-y(t))^{1/2}w}+\\&\qquad
\LPN 2 {\eta_x(\cdot-y(t))^{1/2}w} \norm{\eta_x(\cdot-y(t))^{1/2} w}_{L^{\infty}}^2.
\end{split}
\end{align}
\end{thm}

\begin{proof}
First define
\begin{equation} \label{psi}  \psi(x) = \eta(x) \sech^2(x). \end{equation}

Our aim is to find a function $y$ satisfying the orthogonality condition:
\begin{equation}\label{ortho}
\ip{  u(t,\cdot)-\tanh(\cdot-y(t)) , \psi (x-y(t)) }=0
\end{equation}
for all $t\ge 0$ such that $y(0)=y_0$, where $y_0$ is given by Lemma
\ref{position}.  The existence of such a function, at least initially,
for $t\in [0,T]$, for some $T>0$, is a consequence of Lemma
\ref{position} and the fact $u-\tanh(x+2t)\in C(\RR;H^1(\RR))$.

Now define $w(t,x)$ by
\[ w (x,t) =  u(t,x)- \tanh(x-y(t)), \]
from which we obtain 
\begin{multline}\label{KINKmKdV3} 
w_t + w_{xxx}-2 \partial_x( 3 \tanh^2(x-y) w + 3 \tanh(x-y) w^2+ w^3)=\\ 
(2+\dot y)\sech^2(x-y). 
\end{multline} 

Again by perhaps taking a smaller $T$ if necessary, we can assume for $t\in [0,T]$
\[\LPN 2 {w}\leq 2\epsilon.\]

By (\ref{KINKmKdV3}) and (\ref{ortho}) we obtain:
\begin{multline}\label{change}
\frac {d}{dt} \int \eta(x-y) w^2~dx=\\
\int \bigg[ -3\eta_x(x-y) w_x^2+ \bigg(-\dot y\eta_x(x-y) + \eta_{xxx}(x-y) \\-6\tanh^2(x-y)\eta_x(x-y) + 12\eta(x-y)\tanh(x-y)\sech^2(x-y)\bigg)w^2\\
+(4\eta(x-y) \sech^2(x-y)-8\eta_x(x-y)\tanh(x-y))w^3-3\eta_x(x-y)w^4\bigg]~dx.
\end{multline}
Rewriting the quadratic part of the above equation by using the identity $\sech^2(x)+\tanh^2(x)=1$ numerous times, together with the observation
\begin{equation}\label{uiop}
\eta_{xxx}=-3\eta_x^2+\eta_x,
\end{equation}
along with the trivial identity $\dot y=-2+(\dot y+2)$ we obtain
\begin{multline}\label{change2}
\frac {d}{dt} \int \eta(x-y) w^2~dx=\\
\int \bigg[ -3\eta_x(x-y) w_x^2+ \bigg(-3 +6\sech^2(x-y)+\\24\eta(x-y)\tanh(x-y)\sech^2(x-y)(\eta_x(x-y))^{-1} \bigg)\eta_x(x-y)w^2\\
 -(2 +\dot y)\int \eta_x(x-y) w^2~dx-\int 3(\eta_x(x-y))^2 w^2~dx\\   
+(4\eta(x-y) \sech^2(x-y)-8\eta_x(x-y)\tanh(x-y))w^3-3\eta_x(x-y)w^4\bigg]~dx.
\end{multline}

Now observe
\begin{align*}
\int \left(\eta_x^{1/2}w\right)^2_x~dx&=\int\left[\eta_x w_x^2+\frac{\eta_{xx}^2}{4\eta_x}w^2+\eta_{xx}ww_x\right]~dx\\
&=\int\left[\eta_x w_x^2+\left(\frac{\eta_{xx}^2}{4\eta_x}-\frac12\eta_{xxx}\right)w^2\right]~dx\\
&=\int\left[\eta_x w_x^2+\left({\eta_{x}^2} -\frac14\eta_x\right)w^2\right]~dx,
\end{align*}
where in the last line we used \eqref{uiop} in addition with the identity
\[
\frac{\eta_{xx}^2}{\eta_x}=\eta_x-2\eta_x^2.
\]
We define the quadratic form:
\begin{multline*}
B_{\varepsilon, R}(f):=\int f_x^2 + \bigg(\frac54 -2\sech^2(x)  \\- 8 \sech^2(x)\tanh(x)\cosh^{2}\left(\frac{x-R}{2}\right) \left(1+\varepsilon+\tanh\left(\frac{x-R}{2}\right)\right)  f(x)^2~dx
\end{multline*}
and rewrite the equation (\ref{change2}) as
\[
\begin{split}
\frac {d}{dt} \int  \eta(x-y) & w^2~dx =    -3B_{\varepsilon, R}(\eta_x(x-y)^{1/2}w)
 -(2 +\dot y)\int \eta_x(x-y) w^2~dx
 \\ & +\int (4\eta(x-y) \sech^2(x-y)-8\eta_x(x-y)\tanh(x-y))w^3~dx 
\\ &  -\int 3\eta_x(x-y)w^4 ~dx.
\end{split} 
\]
We observe that $\eta_x$ is positive and hence the last line is non-positive. We will now estimate the cubic term:
\begin{align*} 
\Big|\int (4\eta(x-y) \sech^2(x-y)-&8\eta_x(x-y)\tanh(x-y))w^3 dx\Big| \\
&\lesssim C_R   \int \eta_x |w|^3 dx 
 \\& \lesssim C_R    \LPN 2 w  \LPN{4}{ \eta_x^{1/2} w }^2 
\\ &\lesssim C_R   \LPN 2 w \HN{1}{ \eta_x^{1/2} w}^2 .
\end{align*}

We turn to   bounding $\abs{\dot y +2}$. Note we have from \eqref{KINKmKdV3} and \eqref{psi}:
\begin{multline}\label{dotyeq}
0= \frac{d}{dt} \langle w, \psi(x-y) \rangle = \int \bigg[w\psi_{xxx}(x-y)\\
-2 ( 3 \tanh^2(x-y) w + 3\tanh(x-y) w^2+ w^3)\psi_x(x-y)\\
+ (2+\dot y) \sech^2(x-y)\psi(x-y)   -\dot y    w(x)\psi_x(x-y)\bigg]~dx.
\end{multline}

Thus we obtain
\begin{align}\label{dotyest}
\begin{split}
\abs{\dot{y}+2}&\lesssim C_R (1+\abs{\dot y+2})\LPN 2 {w}+\LPN 2 {\eta_x^{1/2}w}^2+\LPN 2 {\eta_x^{1/2}w} \norm{\eta_x^{1/2} w}_{L^{\infty}}^2 \\
&\lesssim C_R \LPN 2 {w}+\LPN 2 {w} \norm{\eta_x^{1/2} w}_{L^{\infty}}^2,
\end{split}
\end{align}
where in the last line we use the fact that $\LPN 2 w \ll 1$.

Collecting the above estimates together, we obtain:
\begin{equation}\label{pos}
\frac {d}{dt} \int \eta(x-y) w^2 dx \leq  -3B_{\varepsilon, R}(\eta_x(x-y)^{1/2}w)+\Lambda \LPN 2 w \norm {\eta_x^{1/2}w}_{H^1}^2,
\end{equation}
for some constant $\Lambda$ depending on $R$, which will be a positive number. 

We now compare the quadratic form $B_{\varepsilon, R}$ with the quadratic form $B$ from Appendix B: 
\[B(f)=\int f_x^2 + \bigg(\frac54 -2\sech^2(x)  - 4 \sech^2(x)\tanh(x)\bigg) f(x)^2~dx.\]
The difference $V(x)$ of the potentials in the quadratic forms is 
\[ 4 \sech^2(x)\tanh(x)\left(2\cosh^{2}\left(\frac{x-R}{2}\right) \left(1+\varepsilon+\tanh\left(\frac{x-R}{2}\right)\right)- 1\right).    \]
Observe that 
\begin{equation*}
\cosh^2\left(\frac{x-R}{2}\right)\left(1+\tanh\left(\frac{x-R}{2}\right)\right)= \frac{1}{2}(e^{x-R}+1)
\end{equation*}
hence the difference $V$ can be bounded by 
\[ \abs{V} \le  8 \varepsilon \sech^2(x) \cosh^2\left(\frac{x-R}2\right) + 4\sech^2(x) e^{x-R}   
\le  16 \varepsilon e^{R} + 8 e^{-R}.
 \]
Thus we obtain
\begin{equation} \label{diffquad} 
|B(f) - B_{\varepsilon,R}(f)| \le \left(16 \varepsilon e^R   +4 e^{-R}\right) \LPN 2 f ^2. 
\end{equation}
Now define the modified quadratic form
\begin{equation*}
\hat{B}_{\varepsilon, R}(f):=  B_{\varepsilon, R}(f)+2 e^{R}\left\langle \eta_x^{-1/2}\eta\sech^2, f\right\rangle^2.
\end{equation*}
Observe that
\[  \eta_{R,\varepsilon} - \eta_{R,0} = \varepsilon,\]
and
\begin{equation*}
\cosh\left(\frac{x-R}{2}\right)\left(1+\tanh\left(\frac{x-R}{2}\right)\right)= e^{\frac{x-R}{2}},
\end{equation*}
from which it follows that
\[  (e^{R/2} \eta_x^{-1/2} \eta(x)  -e^{x/2}) \sech^2(x) =   
\sqrt{2} \varepsilon  e^{R/2} \cosh\left(\frac{x-R}2\right) \sech^2(x) \le 2 \varepsilon  
e^R,     \]
which yields the estimate
\[ \abs{e^{R/2} \ip{ \eta_x^{-1/2}\eta\sech^2, f}- \ip{  e^{x/2}\sech^2(x) ,f}} \le 2 \varepsilon e^R\LPN 2 f . \]

By estimate \eqref{diffquad} 
 together with the estimate \eqref{assumption} we obtain: 
\begin{lem} With $R=10$ we have  for all $f\in H^1$
\[ \hat{B}_{e^{-2R},R} (f)  \geq \frac{1}{20}  \HN{1}{ f }^2. \]
\end{lem} 
We now fix $R=10$ and set $\varepsilon:=e^{-2R}=e^{-20}$ -- noting that we
only require the existence of $R$ such that the conclusion of the Lemma holds,
with its size being neither optimal nor important.

   If we for the moment assume that
\[ \sup_{0\le t \le T} \LPN 2 {w(t,\cdot)}  \le \frac1{80 \Lambda}, \]
then it follows from the above lemma, and the orthogonality condition (\ref{ortho}) that we can control the second term in the equation (\ref{pos}) with the first term, which implies
\[\int \eta(x-y(t)) w(t,x)^2  dx \le \int \eta(x-y(0)) w(0,x)^2 dx.\] 

Hence we obtain 
\begin{align*}
\varepsilon  \LPN{2} {w(t,\cdot)}^2\le&    \int \eta (x-y)w(t,x)^2  dx \\\le&
 \int \eta(x-y(0)) w(0,x)^2 dx  \\\le&  (1+ \varepsilon)   \LPN{2}{ w(0,\cdot)}^2 
\end{align*}
and thus 
\[ \LPN{2}{ w(t,\cdot)} \le  2 e^{R} \LPN{2}{ w(0,\cdot)}.  
\]

A continuity argument gives the desired global bound provided (recall $\varepsilon = e^{-2R}$)
\[ \LPN{2}{ w(0,\cdot)} \le  \frac1{120 \Lambda} e^{-R}. \]
\end{proof}

\begin{cor}\label{highernorm}
Suppose $u$ satisfies the conditions in the above Theorem, furthermore assume $u(0,\cdot)- \tanh(\cdot)\in H^s$, where $s$ is a positive integer, then there exist $C>0$ and $N >0$  depending only on $s$  such that
\begin{equation}  \label{highernormineq}
\begin{split} 
\HN{s}{u(t,\cdot)-\tanh(x-y(t))}\le &  C \HN{s}{u(0,\cdot)-\tanh(\cdot)}
\\ & \times \left(1+\HN{s}{u(0,\cdot)-\tanh(\cdot)}\right)^N.
\end{split} 
\end{equation} 
\end{cor}
\begin{proof}
Let $w$ be as in Theorem \ref{stab} and define:
\begin{equation}\label{uwrel}
v(t,x):= w(t,x-6t)^2+2w(t,x-6t)\tanh(x-y(t)-6t)+w_x(t,x-6t),
\end{equation}
i.e. $v$ is a solution to \eqref{KdV}.  

Note as a consequence of infinite conservation laws associated with
the KdV equation (see 
Appendix \ref{conserved}), we have
\begin{align}\label{zut}
\begin{split}
\HN {s} {v(t,\cdot)} &\lesssim \HN {s-1} {v(0,\cdot)}\left(1+\HN {s-1} {v(0,\cdot)}\right)^{N^{\prime}}\\
&\lesssim \HN {s} {w(0,\cdot)}\left(1+\HN {s} {w(0,\cdot)}\right)^{N^{\prime\prime}}
\end{split}
\end{align}
for some positive integers $N^{\prime}$ and $N^{\prime\prime}$.

Then from Remark \ref{wvest} at the end of Section \ref{invmiurasec}, Theorem \ref{stab} and \eqref{zut} we obtain \eqref{highernormineq}.
\end{proof}

We now consider the problem of asymptotic stability of the mKdV equation near a kink.  We will require an additional weight function:
\[\phi_{x_0,A}(t,x)=\phi(t,x)=1+\tanh\left(\frac{x-x_0+\gamma t}{A}\right).\]  

\begin{prop}  \label{stableftright2} 
Let $\gamma < 6$, then there exists $\delta, A>0$  such that if  $u$ is the  solution to \eqref{uKINKmKdV} with initial data   satisfying $u(0,\cdot) - \tanh(\cdot) \in H^1(\mathbb{R})$  and $\LPN 2 {u(0,\cdot)-\tanh(\cdot)}<\delta$, and $x_0\in \RR$ we have the bounds
\begin{equation}\label{rightofsol2}
\int \eta(x-y(t))\phi_{A,x_0}(t,x) w(t,x)^2~dx  \lesssim \int \eta(x-y(0)) \phi_{A,x_0}(0,x)w(0,x)^2~dx ,
\end{equation}
where $t>0$, $w:=u-\tanh(\cdot-y)$ and $y$ references to the continuous function constructed in Theorem \ref{stab}.  Moreover, we have the following smoothing estimate:
\begin{equation}\label{othersmooth}
\int_0^{\infty} \HN{1}{(\eta(x-y)\phi_{A,x_0})_x^{1/2}w}^2~dt\lesssim \int
\eta(x-y(0)) \phi_{A,x_0}(0,x)w(0,x)^2~dx.
\end{equation} 
\end{prop}
\begin{proof}
Using the shorthand $\eta=\eta(x-y(t))$, $\eta_x=\eta_x(x-y(t))$, \dots, $\sech^2=\sech^2(x-y(t))$, \dots we obtain:
\[
\begin{split}
\frac {d}{dt} \int & \phi\eta w^2~dx = -B_{\varepsilon,
  R}(\left(\phi\eta_x\right)^{1/2}w)\\ 
& + \int \bigg[-(2 +\dot
  y)\phi\eta_x w^2 + \phi\left(4\eta
  \sech^2-8\eta_x\tanh\right)w^3-3\phi\eta_xw^4 
\\ & + \eta
  \bigg(-3\phi_x w_x^2+ \left(\gamma\phi_x + \phi_{xxx}
  -6\tanh^2\phi_x\right)w^2-8\phi_x\tanh w^3-3\phi_x w^4\bigg) 
\\ &  +
  4\left(\phi_{xx}\eta_{x}+\phi_{x}\eta_{xx}\right)w^2 + (2+\dot
  y)\phi\eta\sech^2w\bigg]~dx
\end{split}
\]
We first note that
\[
\begin{split} 
\left(\gamma  -6\tanh^2(x-y)\right)\phi_x(t,x)
=&\! \left((\gamma-6)+6\sech^2(x-y)\right)\phi_x(t,x)\!\\
 \leq &  (\gamma-6)\phi_x(t,x)  
\\ &+ C A^{-1}\phi (t,x)\eta_x(x-y),
\end{split}
\]
where we used the fact that $\phi_x\lesssim A^{-1}\phi$.
We also have the estimate
\begin{equation*}
\int \eta(x-y)\phi_x(t,x)\tanh(x-y)) w^3 \lesssim \LPN{2}{w}\HN{1}{(\eta(x-y)\phi_x(t,\cdot))^{1/2}w}^2.
\end{equation*}
Thus if we assume $\LPN{2}{w}$ to be suitably small and $A$ to be large,  then by the above estimates and the arguments in Theorem \ref{stab} we obtain:
\begin{align}\label{rightright}
\begin{split}
\frac {d}{dt} \int \phi \eta w^2~dx \leq &  -\kappa \HN{1}{(\phi\eta)_x^{1/2}w}^2\\
& +  c \HN 1 {\eta_x w}^2  \HN{1}{(\phi\eta)_x^{1/2}w}^2
\\
& +\int \bigg[4\left(\phi_{xx}\eta_{x}+\phi_{x}\eta_{xx}\right)w^2\\
 &+(2+\dot y)\phi\eta\sech^2(x-y)w \bigg]~dx
\\ &+C\ip{\phi^{1/2} w,\eta(x-y)\sech^2(x-y)}^2.
\end{split}
\end{align} 

The plan is to integrate \eqref{rightright} to obtain our claim but first we will need estimate the last two terms.

First note for large $A$ we have the following simple estimates
\begin{align*}
\abs{\phi(t,x)-\phi(t,y(t))}\sech(x-y) &\lesssim A^{-1}e^{-2\abs{(y-x_0+\gamma t)/A}},\\
\phi(t,x)^{-1/2}&\lesssim e^{\abs{(x-x_0+\gamma t)/A}},\quad\text{and}\\
\sech(x-y)&\lesssim \eta_x(x-y).
\end{align*}
Applying the above estimates we obtain
\begin{align}\label{orthest1}
\begin{split}
\abs{\int \phi\eta\sech^2(x-y)w~dx} &= \abs{\int \left(\phi(t,x)-\phi(t,y)\right)\eta\sech^2(x-y)w~dx}
\\
&\lesssim A^{-1}e^{-2\abs{(y-x_0+\gamma t)/A}}  \LPN{2}{(\phi\eta_x)^{1/2}w}\times\\&\qquad\LPN{2}{\eta_x^{1/2}\phi^{-1/2}}\\
&\lesssim A^{-1}e^{-\abs{(y-x_0+\gamma t)/A}} \LPN{2}{(\phi\eta_x)^{1/2} w}.
\end{split}
\end{align}

By (\ref{dotyeq}) we have
\begin{align}\label{orthest3}
\begin{split}
\abs{\dot y +2}&\lesssim \int \sech^2(x-y)\left(\abs{w}+\abs{w}^3\right)\\
& \lesssim \LPN{2}{(\phi\eta_x)^{1/2}w}\LPN{2}{\eta_x^{1/2}\phi^{-1/2}}\left(1+\LPN{\infty}{\eta_x^{1/2}w}^2\right)\\&\lesssim e^{\abs{(y-x_0+\gamma t)/A}}\LPN{2}{(\phi\eta_x)^{1/2}w}\left(1+\HN{1}{\eta_x^{1/2}w}^2\right).
\end{split}
\end{align}
Combining (\ref{orthest1}),  (\ref{orthest3}) we get
\begin{multline*}
\abs{\int(2+\dot y)\phi\eta\sech^2(x-y)w~dx}\lesssim A^{-1}\LPN{2}{(\phi\eta_x)^{1/2}w}^2\bigg(1+\HN{1}{\eta_x^{1/2}w}^2\bigg),
\end{multline*}
and similarly for the last term we get
\[ 
\begin{split}
\ip{\phi^{1/2} w,\eta(x-y)\sech^2(x-y)}^2 \lesssim & \LPN{\infty}{\left(\phi(t,\cdot)^{1/2}-\phi(t,y(t))^{1/2}\right)\sech^{1/2}(\cdot-y)\eta}^2\\&\times 
\LPN{2}{\phi^{1/2}\sech^{1/2}(\cdot-y)w}^2\LPN{2}{\sech(\cdot-y)\phi^{-1/2}}^2 
\\\lesssim  & A^{-1}\HN{1}{(\eta_x\phi)^{1/2} w}^2.
\end{split}
\]
Then from the above estimates, if we assume $A$ to be suitably large we obtain
\begin{equation*}
\frac {d}{dt} \int \phi\eta w^2~dx \lesssim \LPN{2}{(\phi\eta)_x^{1/2}w}^2\HN{1}{\eta_x^{1/2}w}^2.
\end{equation*}

Thus from Gronwall's inequality we have
\begin{equation*}
\LPN{2}{\phi(t,x)^{1/2}w(t,x)}^2\lesssim \LPN{2}{\phi(0,x)^{1/2}w(0,x)}\exp\left(\int_0^t\HN{1}{\eta_x^{1/2}w}^2\right).
\end{equation*}

The claim then follows as a consequence of (\ref{virial}).
\end{proof}
As a consequence of the  Proposition \ref{stableftright2} and Theorem \ref{stab} we obtain the following theorem.

\begin{thm}\label{mKdVAsymStabL2}
Let $ \gamma < 6$. Then there exists  $\delta_{\gamma}>0$   such that if $u$ is a  solution to \eqref{mKdV} 
with initial data  $u_0$, satisfying $u_0 - \tanh(x) \in H^1(\mathbb{R})$  and $\LPN 2 {u(0,\cdot)-\tanh(x)}<\delta_{\gamma}$,  
\begin{equation}\label{asymstab}  
 \lim_{t \to \infty} \norm{ u(t,.) - \tanh(x-y(t)) }_{L^2( (-\gamma t , \infty) ) } =  0 
\end{equation} 
where $y:[0,\infty)\to\RR$ refers to the continuous function constructed in Theorem \ref{stab}.
\end{thm}

Making use of the Miura transformation to relate mKdV near the kink 
with KdV near zero,  we will  replace $L^2$ in the statement of the above theorem with $H^s$ for any non-negative integer $s$.  Specifically we have, denoting again $w= u-\tanh(.-y(t))$:

\begin{cor}\label{mKdVAsymStab}
Let $ \gamma < 6$ and $s$ any positive integer. Then there exists  $\delta_{\gamma}>0$   such that if $u$ is a  solution to \eqref{uKINKmKdV} 
with initial data  $u_0$, satisfying $u_0 - \tanh(\cdot) \in H^s(\mathbb{R})$  and $ \LPN 2 { u(0,\cdot)-\tanh(x)} <\delta_{\gamma}$,  
\begin{equation}\label{asymstab2}  
 \lim_{t \to \infty} \norm{ u(t,\cdot) - \tanh(x-y(t)) }_{H^{s}( (-\gamma t , \infty) ) } =  0 
\end{equation} 
where $y:[0,\infty)\to\RR$ refers to the continuous function constructed in Theorem \ref{stab}.  Moreover we have the smoothing estimate
\begin{equation}\label{othersmooth2}  
\int_0^\infty \HN{s+1}{\rho_x(t,\cdot+6t)^{1/2}w(t,\cdot)}^2~dt\leq C\HN s{\rho(0,\cdot)^{1/2}w(0,\cdot)}^2.
\end{equation}
where C depends on $\gamma$ and $\HN s{u(0,\cdot)-\tanh(\cdot)}$, and $\rho$ is defined as
\[\rho(x,t)=1+\tanh\left(\frac{x-x_0+(\gamma-6) t}{A}\right),\]
for some large constant $A>0$.
 \end{cor}
\begin{proof}
Note that the absolute values of the derivatives of $\rho$ are bounded above by a constant multiple of $\rho$.  The same property  is also true for the function $\rho_x$.  This property of $\rho$ and $\rho_x$ will be used extensively below without further comment.  

 Define $v$ as in \eqref{uwrel}, hence  $v$ is a solution to \eqref{KdV}.  Fixing $t\ge0$, observe from \eqref{uwrel},
Lemma \ref{techlem} we have for $f:=\rho^{1/2}$ or $f:=\rho_x^{1/2}$ the following estimate
\begin{align}\label{wtov}
\begin{split}
\HN{s-1}{f(t,\cdot)v(t,\cdot)}\lesssim&\HN{s-1}{f(t,\cdot+6t)w(t,\cdot)^2}+\\&\HN{s-1}{f(t,\cdot+6t)w(t,\cdot)\tanh(\cdot-y)}+\\&\HN{s-1}{f(t,\cdot+6t)w_x(t,\cdot)}\\
\lesssim& \left(1+ \HN {s-1} {w(t,\cdot)}\right) \HN{s}{f(t,\cdot+6t) w(t,\cdot)}\\
\leq &C \HN{s}{f(t,\cdot+6t) w(t,\cdot)},
\end{split}
\end{align} 
for all integers $s\ge 1$, where $C$ depends on $\HN {s-1} {w(t,\cdot)}$. 

Similarly we also have the estimate
\begin{align}\label{vtow}
\begin{split}
\HN{s}{f(t,\cdot+6t)w(t,\cdot)}\lesssim& \LPN 2{f(t,\cdot+6t) w(t,\cdot)}+\HN{s-1}{f(t,\cdot+6t) w_x(t,\cdot)}\\
\lesssim &\LPN 2{f(t,\cdot+6t) w(t,\cdot)}+\HN{s-1}{f(t,\cdot) v(t,\cdot)}+\\
&\HN{s-1}{f(t,\cdot+6t) w(t,\cdot)^2}+\\
&\HN{s-1}{f(t,\cdot+6t) w(t,\cdot)\tanh(\cdot-y)}\\
\lesssim & \left(1+\HN {s-1}{w(t,\cdot)}+\HN {1}{w(t,\cdot)}\right)\times\\&\HN{s-1}{f(t,\cdot+6t) w(t,\cdot)}+\HN{s-1}{f(t,\cdot) v(t,\cdot)}\\
\leq & C \left( \HN{s-1}{f(t,\cdot+6t) w(t,\cdot)}+\HN{s-1}{f(t,\cdot) v(t,\cdot)} \right),
\end{split}
\end{align}
for all integers $s\ge 1$, where $C$ depends on $\HN {s-1} {w(t,\cdot)}+\HN {1} {w(t,\cdot)}$.

The inequalities \eqref{wtov} and \eqref{vtow} will essentially allow to shift our focus from a study of mKdV near a kink to that of KdV in a neighbourhood of zero.  In particular, note that by Theorem \ref{stab} and Corollary \ref{highernorm}, the constants in \eqref{wtov} and \eqref{vtow} depend only on the initial data $\HN {s-1} {u(0,\cdot)-\tanh(\cdot)}$ and $\HN {s-1} {u(0,\cdot)-\tanh(\cdot)}+\HN {1} {u(0,\cdot)-\tanh(\cdot)}$ respectively. 

Now consider the case $s=1$.  Below $C$ will denote a positive constant depending on $\HN {1} {u(0,\cdot)-\tanh(\cdot)}$ and $\gamma$, which may change from line to line. 

A simple energy estimate yields 
\begin{equation*}
\frac {d}{dt} \int \rho v^2 dx =\int \rho_t v^2+\rho_{xxx}v^2-3\rho_x v_x^2-4\rho_x v^3~dx.
\end{equation*}
Note that replacing $\rho$ with $1$ we recover the $L^2$ conservation law for KdV.

Also, we have the simple estimate
\begin{align*}
\int \rho_x v^3~dx &\leq \LPN 2 v \LPN 2 {\rho_x^{1/2} v} \LPN {\infty} {\rho_x^{1/2} v}\\
&\lesssim \LPN 2 {v(0,\cdot)} \LPN 2 {\rho_x^{1/2} v} \HN 1 {\rho_x^{1/2} v}\\
&\leq C \left(\varepsilon^{-1} \LPN 2 {\rho_x^{1/2} v}^2 +\varepsilon \HN 1 {\rho_x^{1/2} v}^2\right),
\end{align*}
for any $\varepsilon>0$.

Hence from the above estimates
\begin{align*}
\frac {d}{dt} \int \rho(t,x) v^2(t,x)~dx\leq&- 2 \LPN{2}{\rho_x(t,\cdot)^{1/2}v_x(t,\cdot)}^2
 \\&\qquad+  C \HN{1}{\rho_x(t,\cdot+6t)^{1/2} w(t,\cdot)}^2.
\end{align*}

Therefore from \eqref{othersmooth} we obtain 
\begin{align*}
&\int \rho(t,x) v^2(t,x)~dx+2\int_0^{\infty} \LPN{2}{\rho_x(t^{\prime},\cdot)^{1/2}v_x(t^{\prime},\cdot)}^2~dt^{\prime}\lesssim \\&\qquad \int \rho(0,x) v(0,x)^2~dx+ C\int \rho(0,x)w(0,x)^2~dx.
\end{align*}

Then from the above inequality, \eqref{vtow} (with $f=\rho^{1/2}$), and \eqref{asymstab}, we obtain \eqref{asymstab2} for $s=1$.  Similarly from the above inequality, \eqref{vtow} (with $f=\rho_x^{1/2}$), and \eqref{othersmooth}, we obtain \eqref{asymstab2} for $s=1$.  

We now will provide a sketch of the proof for $s>1$.  We proceed by induction, assuming as our inductive hypothesis that \eqref{asymstab2} and \eqref{othersmooth2} holds for a given positive integer $s$. Below $C$ will denote a positive constant depending on $\HN {s+1} {u(0,\cdot)-\tanh(\cdot)}$ and $\gamma$, which may change from line to line. 

From \eqref{conslaw} we have
\begin{equation}\label{iterhiest}
\frac{d}{dt} \int \rho  T^{(s)}dx =\int \rho_t T^{(s)}+\rho_x X^{(s)}~dx.
\end{equation} 
Observe that from the two monomials $2\partial_x^su\partial_x^{s+2}u$ and $-(\partial_x^{s}u)^2$ in $X^{(s)}$ we recover (after a couple of integration by parts) the terms 
\begin{equation}\label{goodterm}
-3\rho_x(\partial_x^{s+1}u)^2+\rho_{xxx}\left(\partial_x^s u\right)^2,
\end{equation}
in the integrand on the right hand side of \eqref{iterhiest}.

We now proceed in a similar manner to the case of $s=1$, using extensively the properties of $X^{(s)}$ and $T^{(s)}$ as stated in Appendix \ref{conserved}.  In this way, one can show that 
\begin{align*}
\int \rho_t T^{(s)}+\rho_x X^{(s)}+3\rho_x(\partial_x^{s+1}u)^2~dx &\lesssim \left(1+\HN{s}{v}^{s+1}\right)\HN{s}{\rho_x^{1/2} v}^2\\&\leq C \HN{s}{\rho_x^{1/2} v}^2.
\end{align*}
Integrating \eqref{iterhiest} with respect to $t$, and using our induction hypothesis together with \eqref{wtov} and \eqref{vtow} leads to
\begin{multline*}
\int \rho(t,x)  T^{(s)}(t,x)~dx+3\int_0^\infty \LPN {2}{\rho_x(t^{\prime},\cdot)^{1/2}\left(\partial_x^{s+1} v\right)(t^{\prime},\cdot)}^2~dt^{\prime}\leq \\C \HN{s}{\rho(0,\cdot)^{1/2} w(0,\cdot)}^2+ \int \rho(0,x)  T^{(s)}(0,x)~dx.
\end{multline*}

We observe that the terms on the left hand side of the above equation
resulting from lower order terms in $T^{(s)}$ can be bounded by a constant multiple of
\[\left(1+\HN{s}{v}^{s}\right)\HN{s-1}{\rho^{1/2} v}^2\leq C \HN{s-1}{\rho^{1/2} v}^2.\]
Similarly, the terms on the right hand side 
resulting from lower order terms in $T^{(s)}$ can be bounded by $C \HN{s-1}{\rho(0,\cdot)^{1/2} v(0,\cdot)}^2$.
Thus we obtain
\begin{multline*}
\int \rho(t,x)  \left(\partial_x^s v\right)^2(t,x)~dx+3\int_0^\infty \LPN {2}{\rho_x(t^{\prime},\cdot)^{1/2}\left(\partial_x^{s+1} v\right)(t^{\prime},\cdot)}^2~dt^{\prime}\leq \\ \int \rho(0,x)  \left(\partial_x^s v\right)^2(0,x)~dx+\\C \left(\HN{s}{\rho(0,\cdot)^{1/2} w(0,\cdot)}^2+\HN{s-1}{\rho (t,\cdot)^{1/2}v(t,\cdot)}^2+\HN{s-1}{\rho(0,\cdot)^{1/2} v(0,\cdot)}^2\right).
\end{multline*}
 Then applying \eqref{wtov} and \eqref{vtow}, together with our induction hypothesis, we obtain \eqref{asymstab2} and \eqref{othersmooth2} for $s+1$.
\end{proof}


\section{Existence of weak solutions to the Korteweg--de Vries equation with initial data in $H^{-1}$}

With the help of Theorem \ref{strongsol}, Lemma \ref{localbd} and Lemma \ref{KATOSMOOTH}, we will now prove the existence of weak $L^2$ solutions to the IVP (\ref{KINKmKdV}).  

\begin{prop}\label{weakKINKmKdV}
For any $v_0\in L^2$, there exists a weak solution $u=v+Q$ to (\ref{KINKmKdV}) satisfying
\begin{align}
v\in C_{\omega}([0,\infty); L^2),\label{prop1}\\
v_x\in L^2([0,T]\times[-R,R]) ~~\text{for any}~~ R,T<\infty,\label{prop2}\\
\LPN{2}{v(t,\cdot)} \lesssim \LPN{2}{v(0,\cdot)}+t^{1/2} ~~\text{for any}~~ t\in[0,\infty),\label{prop4}\\
 v(t,\cdot) \to v_0 \qquad \text{ in } L^2 ~~\text{as}~~t\downarrow 0.\label{prop3}
\end{align}
Furthermore there exists a $\delta>0$ such that if $\LPN 2 {v_0}<\delta$ then there exists a continuous function $y:\RR\to\RR$ such that if we write $u=w+\tanh(\cdot+y(t))$, we have
\begin{equation}\label{lastpro}
\LPN {2}{w}\lesssim \LPN 2 {v_0}.
\end{equation}
\end{prop}
\begin{proof}
Let $v_0^{(j)}\in H^{1}$ be a sequence such that $v_0^{(j)}\rightarrow
v_0$ in $L^2$, and $\LPN{2}{v_0^{(j)}}=\LPN{2}{v_0}$.  Define $v^{(j)}\in
C([0,\infty); H^{1})$ to be the solution to (\ref{KINKmKdV}) with
  $v^{(j)}(0,\cdot)=r_{0,j}$, corresponding to Theorem \ref{strongsol}.  
  
   If in addition we have $\LPN 2 {v_0}<\delta$, and we write 
  \[u^{(j)}(t,x)=v^{(j)}(t,x)+Q(t,x)=w^{(j)}(t,x)+\tanh(x+y^{(j)}(t)),\]
  where $y^{(j)}$ is defined as in Theorem \ref{stab}, then using \eqref{dotyest3}, \eqref{apriorimkdv}, and \eqref{virial} we obtain a uniform bound of $y^{(j)}$ in $H^1([0,T])$, and thus by Morrey's inequality we have a uniform bound of $y^{(j)}$ in $C^{0,1/2}([0,T])$, for any fixed $T>0$.  By the Azelà-Ascoli theorem, and a suitable diagonal argument we can construct a subsequence $(v^{(N_j)})$ such that for all $T>0$, $y^{(N_j)}$ converges uniformly to some continuous function $y:\RR^+\rightarrow \RR$.  Moreover from \eqref{apriorimkdv}, we have for any $t\geq 0$, there exists a $k$ such that if
$j>k$
\begin{equation}\label{unibd}
\LPN {2}{u^{({N_j})}(t,\cdot)-\tanh(\cdot-y(t))}\lesssim \LPN 2 {v_0}.
\end{equation}

  Now applying an almost identical argument to the one given in
  \cite{MR759907} to construct weak $L^2$ KdV solutions -- here the
  smoothing estimate is replaced by \eqref{KATOSMOOTH}, and $L^2$
  conservation replaced by \eqref{badL2bd} -- we obtain a subsequence
  $(v^{(N^{\prime}_j)})$ such that for any $R,T>0$ the sequence converges
  weakly in $L^2([0,T]);H^1([-R,R]))$, strongly in
  $L^2([0,T]\times[-R,R]))$ and weak-* in $L^\infty([0,\infty);L^2)$ to
    a limit $v$ satisfying (\ref{prop1}-\ref{prop4}), and solves
    \eqref{KINKmKdV} in the distributional sense. 

In order to prove \eqref{prop3} we set $\tilde v= v\sqrt{1+\frac1{10}Q}$, and observe that $\tilde v$ is continuous at  $t=0$ if and only \eqref{prop3} is satisfied.  Note that weak continuity of $\tilde v$ in $t$ follows from weak continuity of $v$.  Estimating we obtain
\begin{align*}
\LPN 2{\tilde v(t,\cdot)-\tilde v(0,\cdot)}^2&=\LPN 2{\tilde v(t,\cdot)}^2+\LPN 2{\tilde v(0,\cdot)}^2-2\ip{\tilde v(t,\cdot),\tilde v(0,\cdot)}\\
&\le \left(\LPN 2{\tilde v(t,\cdot)}^2-\LPN 2{\tilde v(0,\cdot)}^2\right)+2\LPN 2{\tilde v(0,\cdot)}^2-2\ip{\tilde v(t,\cdot),v_0}\\
&=\left(\LPN 2{\tilde v(t,\cdot)}^2-\LPN 2{\tilde v(0,\cdot)}^2\right)+2\ip{\tilde v(0,\cdot)-\tilde v(t,\cdot),\tilde v(0,\cdot)}.
\end{align*}
Then from \eqref{equality} and the weak continuity of $\tilde v$ we obtain \eqref{prop3}.

Finally, note \eqref{lastpro} is a simple consequence of \eqref{unibd}.
\end{proof}

We will now construct weak $H^{-1}$ solutions to the Korteweg--de
Vries equation.  Using the scaling symmetry, we may restrict to small
initial data in $H^{-1}$.

\begin{prop}\label{lastprop}
For any $u_0\in H^{-1}$ satisfying $\HN{-1}{u_0}\leq \epsilon$, for $\epsilon>0$ chosen suitably small, there exists a weak solution $u$ to (\ref{KINKmKdV}), and a continuous function $y:[0,\infty)\to\RR$  satisfying
\begin{align}
u\in C_{\omega}([0,\infty); H^{-1}),\label{propp1}\\
u\in L^2([0,T]\times[-R,R]) ~~\text{for any}~~ R,T<\infty,\label{propp2}\\
\HN{-1}{u}\lesssim \HN{-1}{f},\label{propp3}\\
u(t,\cdot) \to u_0 \quad \text{ in } H^{-1} ~~\text{as}~~t\downarrow 0.\label{propp4}
\end{align}
\end{prop}
\begin{proof}
Define $v_0$ such that $F(v_0)=(u_0,0)$, where $F$ is defined as in Theorem \ref{inverse}.  Then by Proposition \ref{weakKINKmKdV}, there exists a weak solution $\tilde u$ to the mKdV equation corresponding to initial data $v_0+\tanh(\cdot)$.  Let $u$ be map obtained by applying the Galilean transformation ($h=6$) to $\miura(\tilde u)$. It is then easy to check that $u$ satisfies (\ref{propp1}-\ref{propp4}).  What remains to be shown is that $u$ satisfies (\ref{KdV}) in a distributional sense, which is equivalent to $\miura(\tilde u)$ satisfying (\ref{KdV}) in a distributional sense -- this is the subject of Lemma \ref{distkdvlem} below.
\end{proof}

\begin{lem}\label{distkdvlem}
Let $v_0\in L^2$, and suppose $\tilde u$ is a weak solution to (\ref{mKdV}), satisfying the properties (\ref{prop1}-\ref{prop3}), then $u:=\partial_x(\tilde u)+(\tilde u)^2$ satisfies (\ref{KdV}), in a distributional sense, i.e.
\begin{equation*}
\int_{\RR^2}\left[-u\varphi_t - u\varphi_{xxx} + 3 u^2\varphi_x\right]~dt~dx  =0.
\end{equation*}
for all $\varphi\in C^{\infty}_0$.
\end{lem}

For a proof of the above lemma we refer the reader to the papers \cite{MR990865} and \cite{MR2189502}.  

By utilising the scaling symmetry of the KdV equation and Proposition
\ref{lastprop}, one easily obtains existence of weak solutions of
Theorem \ref{mainTheorem}.


\section{A priori bounds and soliton stability}

\begin{repthm}{aprioriTHM}
Suppose $u\in C([0,\infty);H^s(\RR))$ is a solution to (\ref{KdV}), for some $s\geq -\frac{3}{4}$, then
\begin{equation}
\HN{-1}{u(t,\cdot)} \lesssim  \HN{-1}{u_0}+ \HN{-1}{u_0}^3   ~~\text{ for }~~ t\in[0,\infty).
\end{equation}
\end{repthm}
\begin{proof}
First consider the case when $s=0$. By scaling, the problem reduces to showing that for all solutions $u\in C([0,\infty);L^2(\RR))$ to (\ref{KdV}) satisfying $\HN{-1}{u}\leq \epsilon$ for some suitably chosen $\epsilon>0$, we have
\begin{equation}\label{apriori3}
\HN{-1}{u(t,\cdot)} \lesssim 1  ~~\text{for any}~~ t\in[0,\infty).
\end{equation}
From Theorem \ref{inverse}, Theorem \ref{strongsol} and the well-posedness theory of the KdV equation, it follows that there exists a solution $\tilde{u}\in C([0,\infty);H^{1}(\RR))$ to (\ref{mKdV}), such that the Galilean transformation ($h=6$) of $\miura(\tilde{u})$ is   $u$. Assuming we chose $\epsilon$ sufficiently  small, then as a consequence of Lemma (\ref{localbd}) and Theorem \ref{stab}, we obtain  (\ref{apriori3}).

 The general case when $s\geq -\frac{3}{4}$ can be proven via approximation.
\end{proof}

\begin{repthm}{KdVstab} 
There exists an $\varepsilon >0$ such that 
if $u\in C([0,\infty);H^s(\RR)\cap H^{-3/4}(\RR))$ is a solution to (\ref{KdV}), for some integer $s\geq-1$, satisfying $\HN{-1}{R_c-u_0}<\varepsilon c^{1/4}$ for some $c>0$, then there is a continuous function $y:[0,\infty)\to\RR$ such that 
\[ \HN{s}{ u - R_{ c}(x-y(t))} \leq \gamma_{s}(c,\HN{s} {R_c-u_0})\] 
for any $t\geq 0$, where $\gamma_{s}:(0,\infty)\times [0,\infty)$ is a continuous function, 
polynomial in the second variable, which  satisfies $\gamma(\cdot,0)=0$.
\end{repthm} 
\begin{proof}
  The proof follows in a similar manner to that of Theorem
  \ref{aprioriTHM}.  Again, without loss of generality we may assume
  $u_0\in H^1$.  By scaling we may also assume that $c=4$.  Then assuming
  $\HN{-1}{R_c-u(0,\cdot)}$ to be suitably small, and making use of
  the arguments in Section \ref{invmiurasec}, we may link the KdV IVP with initial
  data $u(0,\cdot)$ to the mKdV IVP with initial data $\tilde
  u_0:=\lambda\tanh(\lambda\cdot)+v_0$, for some $\lambda\approx 1$,
  such that $v_0\in H^{s+1}$ and $\LPN{2}{v_0}\lesssim
  \HN{-1}{R_4-u_0}$.  By scaling on the mKdV side, we can
  assume $\lambda=1$.  The conclusion then follows from Theorem
  \ref{strongsol}, the well posedness theory of the KdV equation,
  Theorem \ref{stab}, and Corollary \ref{highernorm}.  
\end{proof}

Making use of Theorem \ref{mKdVAsymStabL2}, Corollary \ref{mKdVAsymStab}, and following a similar argument to that given above we obtain:

\begin{repthm}{KdVAsymStab} 
Given real $\gamma>0$ and integer $s\geq -1$, there exists an
  $\varepsilon_{\gamma} >0$ such that if $u\in
  C([0,\infty);H^s(\RR)\cap H^{-3/4}(\RR))$, is a solution to
  (\ref{KdV}), satisfying 
\[ \HN{-1}{R_c-u_0}<\varepsilon_{\gamma}  c^{1/4} \]
for $c>0$, then there is a continuous function
  $y:[0,\infty)\to\RR$ and $\tilde c>0$ such that
\begin{equation*}  
 \lim_{t \to \infty} \norm{ u - R_{\tilde c}(x-y(t)) }_{H^{s}( (\gamma t , \infty) ) } =  0 
\end{equation*} 
for any $t\geq 0$.  Moreover we have the bound $\abs{c-\tilde  c}\lesssim  c^{\frac34} \HN{-1}{R_c-u_0}   $.
\end{repthm}

\appendix
\section{Schrödinger operators with rough potentials}\label{appenSchr}

In this section we collect a couple of useful results concerning
Schrödinger operators with distributional $H^{-1}$ potentials.  This
subject was partially studied by Kapeller et al.\ \cite{MR2267286} in a
direction  similar to ours as discussed above.  The Miura map is a
central part of the integrable structure of KdV and mKdV, and hence it
provides a link to Schrödinger operators and inverse
scattering. Typically the inverse scattering methods requires
integrability of the potentials and even some decay. Nevertheless
trace identities allow to express the $L^2$ norm (as well as higher
norms) in terms of the scattering data.  This is relation has been
used by Deift and Killip \cite{MR1697600} to study the spectral
density for $L^2$ potentials. The available results indicate that the
spectrum of $L^2$ potentials is a highly non-trivial and difficult
object. The failure of surjectivity of the Miura map in the work of Kappeler
et al.\ can be seen as a shadow of this complexity.

Here we aim for something considerably simpler: our main 
spectral object is the ground state energy, which is much more robust. 
We start by noting that there is a factorisation of the Schrödinger operator
\begin{equation*} 
H_q:= -\partial_{xx}^2 + q = -(\partial_x + r)(\partial_x -r),  
\end{equation*} 
if $q$ satisfies the Ricatti equation $q = r_x +r^2$.  Moreover, with $\phi=  e^{\int_0^x r dx}$
\[ \partial_{xx} \phi   = \phi ( r_x+r^2) = \phi q 
\] 
and $\phi$  is a non-negative solution to the Schrödinger equation:
\[H_q \phi =0.\] Conversely, if $\phi$ is non-negative and satisfies 
\[ \phi_{xx} + q \phi = 0 \]
then, with 
\[ r  = -\partial_x \ln \phi, \]
we have
\[ r_x +r^2= -\frac{\phi_{xx}}{\phi} = q. \]

\begin{lem} Let $q \in H^{-1}$. Then the Schrödinger operator
\[  \phi \to H_q \phi= -\phi_{xx} + q\phi \]
has a unique self adjoint, semi-bounded below extension.  
\end{lem} 
\begin{proof}
Note that it suffices to show $H_q$ is semi-bounded below: the unique self adjoint, semi-bounded below extension follows by Friedrichs' construction \cite{MR0493420}.  We now turn to the bound from below.

Using a combination of duality, a product estimate, Gagliardo-Nirenberg inequality and Young's inequality we obtain
\begin{align*}
\int qf^2&\leq \HN{-1}{q}\HN{1}{f^2}\\
&\lesssim \HN{-1}{q}\HN{1}{f}\LPN{\infty}{f}\\
&\lesssim \HN{-1}{q}\HN{1}{f}^{3/2}\LPN{2}{f}^{1/2}\\
&\lesssim C^{-4/3}\HN{1}{f}^{2} + C^4 \HN{-1}{q}^4\LPN{2}{f}^{2}.
\end{align*}

Thus taking $C$ large we obtain
\begin{align*}
  \ip{Tf,f}&=  \int f_x^2+ qf^2\\
  &\gtrsim -\left(1+\HN{-1}{q}^4\right) \LPN{2}{f}^2.
\end{align*}
\end{proof}

\begin{lem}\label{maximum}  
Let $q_i \in H^{-1}(a,b)$ and suppose that $\phi,\psi\in H^1(a,b)$ are strictly positive functions satisfying
\[ -\phi'' + q_1 \phi = 0, \qquad -\psi'' + q_2 \psi = 0, \]
 and $q_2 \le q_1$. Then $\phi/\psi$ has no  interior minimum unless it is constant.  
\end{lem}  
\begin{proof} 
We will proceeding formally, however we note that it is not difficult to make the calculations rigorous, then
\[
  -  \frac{d^2 }{dx^2} \frac{\phi}{\psi} + (q_1-q_2) \frac{\phi}{\psi} 
- 2 \frac{\psi'}{\psi}  \frac{d}{dx} \frac{\phi}{\psi}  = 0 
\]
and since  $q_2\le q_1$ we obtain
\[ - \frac{d^2}{dx^2} \frac{\phi}{\psi} - 2\frac{\psi'}{\psi}
\frac{d}{dx} \frac{\phi}{\psi} \le 0. \] 
We claim that $u(x) = \phi/\psi$ cannot have an
interior positive minimum. We search for a contradiction, and assume that $u(x_0)= M
= \inf_{x\in(a,b)}u(x) $ and $u(a) , u(b) > M$.  We test with
$u_\varepsilon =((M+\varepsilon)-u)_+$; setting $U_\varepsilon = \{ x: M < u < M+\varepsilon\}$ yields 
\[  \int_{U_\varepsilon} (u_\varepsilon)_x^2  -   2u_\varepsilon \frac{\psi'}{\psi} (u_\varepsilon)_x dx \leq 0.\]
By assumption, for $\varepsilon$ 
sufficiently small, the quotient $\psi'/\psi$ is uniformly bounded  by some constant $c>0$. Thus 
\[ \begin{split} \int_{U_\varepsilon} u_x^2  \le & \frac12 \int_{U_\varepsilon} (u_\varepsilon)_x^2 dx + c^2 \int_{U_\varepsilon}  u_\varepsilon^2 dx 
\\ \le & \left(\frac12 + c^2 |U_\varepsilon|^2\right) \int_{U_\varepsilon} u_x^2 dx 
\end{split} 
\]
and hence, since the left hand side is nonzero, 
\[ |U_\varepsilon| \ge \frac1{2c}, \]
Letting $\varepsilon$ tend to $0$ we obtain a contradiction. 
\end{proof}

\begin{lem} \label{asymp} Suppose $\lambda>0$, $r \in L^2_{loc} $, $q\in H^{-1}$ and 
\[ r_x+r^2 = \lambda^2+ q  \]
Then either 
\begin{equation}\label{inc1}
r-\lambda \in L^2(0,\infty)\quad \text{ or} \quad r+ \lambda \in L^2(0,\infty)
\end{equation}
and either 
\begin{equation}\label{inc2}
r-\lambda \in L^2(-\infty,0) \quad \text{ or } \quad r+\lambda \in L^2(-\infty,0).
\end{equation}
\end{lem} 
\begin{proof} 
By the symmetry of the problem, it suffices to restrict our attention to (\ref{inc1}).

Since $q\in H^{-1}$, there exists functions $f,g\in L^2$ such that $q=f+g^{\prime}$.  Define $y=r-g$; hence $f$ satisfies 
\begin{equation}\label{yode}  y_x + y^2+ 2gy = \lambda^2+ f - g^2  \end{equation}
in the distribution sense.

Now for a given large $x_0$, we will now investigate the behaviour of $y$ on the interval $[x_0,x_0+1]$. Define $\eta:=e^{2\int_{x_0}^x g}$, $H=\int_{x_0}^x\eta (f-g^2)$ and
\begin{equation}
\tilde y = y\eta - H.
\end{equation}

Thus, $\tilde y$ satisfies
\begin{equation}
\tilde y_x+y^2\eta = \eta\lambda^2.
\end{equation}

Taking $x_0$ to be sufficiently large we may assume $\eta$ to be arbitrarily close to $1$ and $H$ arbitrarily small on the interval $[x_0,x_0+1]$.  More precisely, we can show for a given $\delta>0$, there exists $z\in \RR$ such that if $x_0>z$ then on the interval $[x_0,x_0+1]$
\begin{equation}
\tilde y - y = e_1,
\end{equation}
and
\begin{equation}\label{simy}
\tilde y_x =\lambda^2-\tilde y^2 + e_2
\end{equation}
where the functions $e_1$ and $e_2$  satisfy the bound
\begin{equation}\label{simode}
\abs{e_{\{1,2\}}}\leq \delta\abs{\tilde y}+\delta.
\end{equation}

That is, $\tilde y$ behaves like the non-linear ODE $h^{\prime}=\lambda^2-h^2$, which has a stable fixed point at $\lambda$ and an unstable fixed point at $-\lambda$.  Since $\tilde y\in L^2_{\text{loc}}$, it is then not difficult to show from (\ref{simy}) and (\ref{simode}) that $\abs{y}\rightarrow \lambda$.

Now consider the case when $y\rightarrow \lambda$.  Pick $z\in \RR$ such that $\norm{y-\lambda}_{L^{\infty}[z,\infty)}<\min\{1,\lambda\}$; hence from (\ref{yode}) we obtain
\begin{equation}
\norm{y-\lambda}_{L^2[z,\infty)}\lesssim \frac{1}{\lambda}\bigg(1+\norm{g}_{L^2[z,\infty)}^2+\lambda \norm{g}_{L^2[z,\infty)}+\norm{f}_{L^2[z,\infty)}\bigg).
\end{equation}
Similarly for the case when $y\rightarrow -\lambda$, if we pick $z\in \RR$ such that $\norm{y+\lambda}_{L^{\infty}[z,\infty)}<\min\{1,\lambda\}$ we obtain
\begin{multline}
\norm{y+\lambda}_{L^2[z,\infty)}\lesssim \frac{1}{\lambda}\bigg(1+\norm{g}_{L^2[z,\infty)}^2+\lambda \norm{g}_{L^2[z,\infty)}+\norm{f}_{L^2[z,\infty)}\bigg).
\end{multline}
\end{proof}

\section{Quadratic form estimates}

We consider the quadratic form defined by
\begin{equation}\label{quadl2} 
B(f):=\int f_x^2 + \bigg(\frac54 -2\sech^2(x)  - 4 \sech^2(x)\tanh(x)\bigg) f(x)^2~dx.
\end{equation}

\begin{prop}\label{quadform} 
The quadratic form $B$ satisfies the following inequality
\begin{equation}\label{firstblah}
B(f) + 2  \langle f, e^{\cdot/2}\sech^2(\cdot) \rangle^2  \ge  \frac13 \LPN 2 f ^2  ,
\end{equation}
holds for all $f \in H^1$; moreover we also have the estimate 
\begin{equation}\label{assumption}
B(f) + 2  \langle f, e^{\cdot/2}\sech^2(\cdot) \rangle^2  \ge  \frac1{10}  \HN 1 f ^2.  
\end{equation}
\end{prop}
\begin{rem}
The inequality \eqref{assumption} is actually a simple consequence of \eqref{firstblah}.  A straight forward calculation yields  
\[  2-2 \sech^2(x) - 4  \sech^2(x) \tanh(x) > -2 \]
and hence
\begin{equation}\label{rtz} B(f)\geq \LPN 2 {f_x}^2-2\LPN 2 {f}^2. \end{equation}
Rewriting $B=9B/10+B/10$ and using \eqref{rtz} and \eqref{firstblah} to estimate the first and second term respectively, we obtain \eqref{assumption}.

Note also that the constant $1/10$ is neither optimal nor of any particular importance in the context of the paper, as we will simply require the  existence of a non-negative 
constant. 
\end{rem}
\begin{proof}
  First consider the Schrödinger $H=-\partial_{xx}+V(x)$ operator with
  potential $V(x):= -2\sech^2(x) - 4 \sech^2(x)\tanh(x)$.  A
  celebrated theorem by Lieb and Thirring \cite{lieb2005inequalities}
  gives us a bound on the moments of the bound states energies
  (negative eigenvalues) $e_j$ of $H$:
\begin{equation*}
\sum_j \abs{e_j}^{\gamma}\leq L_{\gamma,1}\int \abs{V(x)}_{-}^{\gamma+n/2}
\end{equation*}
for $\gamma\geq\frac32$, where $\abs{V(x)}_{-}=(\abs{V(x)}-V(x))/2$ and
\begin{equation*}
L_{\gamma,1}=\frac{1}{2\sqrt{\pi}}\Gamma(\gamma+1)/\Gamma\left(\gamma+\frac{3}{2}\right).
\end{equation*}
In particular for $\gamma:=\frac{3}{2}$ we have 
\begin{equation}\label{lieb}
\sum_j \abs{e_j}^{3/2}\leq \frac{3}{16}\int \abs{V(x)}_{-}^{2}= 567/320,
\end{equation}
where the second equality  involves determining the support of 
$\abs{V(x)}_{-}$, and an evaluation of the integral. This was done with the help of Mathematica, but could easily be done by hand. 

It follows immediately that the ground state satisfies the bound
\begin{equation} \label{rangeg} e_0 \geq  -(567/320)^{2/3} > - \frac{3}{2}.\end{equation} 
Now, let $u=\sqrt{2 / \pi}e^{x/2}\sech^2(x)$ -- this is normalised so that the $L^2$ norm of $u$ is $1$.  Then an explicit calculation yields
\[\ip{H(u),u} = -5/4,\]
and thus
\[
-5/4\geq e_0 \geq -(567/320)^{2/3}.
\]

Furthermore from (\ref{lieb}) if we denote the ground state as $v_0$ we have
\begin{align*}
-5/4=H(u)\geq e_0 \abs{\ip{u,v_0}}^2 -(567/320-\abs{e_0}^{3/2})^{2/3}(1-\abs{\ip{u,v_0}}^2),
\end{align*}
hence
\[\abs{\ip{u,v_0}}^2 \geq  \frac{-5/4+(567/320-\abs{e_0}^{3/2})^{2/3}}{e_0+(567/320-\abs{e_0}^{3/2})^{2/3}}\]
Denoting the right hand side by $h(s)$ evaluated at $e_0$; then one can check -- either with the help of a software package 
such as Mathematica, or by hand, with patience -- 
that for $s$ satisfying the bounds \eqref{rangeg},  $h$ has a minimum 
$\frac{1701+\sqrt{1435533}}{3402}$ at $s=-\frac{721489+567 \sqrt{1435533}}{960000}$.
Hence we obtain
\begin{equation}\label{square}  \abs{\ip{u,v_0}}^2 \geq \frac{1701+\sqrt{1435533}}{3402}> \frac{5}{6}.\end{equation} 

Also as a consequence of (\ref{lieb}) and (\ref{rangeg}), we have that for any $v\in H^2$ in the orthogonal complement of $v_0$ 
\begin{equation}\label{second}\ip{H(v),v}\geq -\left(567/320-(5/4)^{3/2}\right)^{2/3}\LPN{2} v^2
\geq - \frac{5}{9} \LPN{2} v^2. 
\end{equation}

Now pick $f\in H^2$ and let $f(x)=a v_0(x)+ g(x)$ be a $L^2$ orthonormal decomposition.  Then applying Young's inequality in the first inequality and orthogonality of $v_0$ and $g$ for the second inequality we have
\begin{align*}
 \ip{f,u}^2=&a^2\ip{v_0,u}^2+2a\ip{v_0,u}\ip{g,u}+\ip{g,u}^2\\ 
\geq & \frac{a^2}2 \ip{v_0,u} - \langle g, u \rangle^2 
\\ \geq & \frac{a^2}{2}\ip{v_0,u}^2- \LPN{2}{g}^2 (1-\ip{v_0,u}^2),
\end{align*}
hence 
\[
\ip{H(f),f}+2\ip{f,u}^2 \geq a^2 \left(e_0+ \ip{v_0,u}^2\right)+\LPN{2}{g}^2 
\left(-\frac59-2(1-\ip{v_0,u}^2)\right). 
\]
The claim \eqref{assumption} follows from the observations 
\[ e_0 + \langle v_0,u \rangle^2 \ge -\frac32 + \frac56 = -\frac23,\text{ and } 
   -\frac59-2(1-\ip{v_0,u}^2)\geq -\frac59 -\frac26 = - \frac8 9,  \]  
since  $\frac54- \frac89 > \frac13$.  
\end{proof}

\section{Higher energies} 
\label{conserved} 

In order to study higher regularity we need to make use of higher order
\emph{polynomial conservation laws} (see \cite{MR0252826} and
\cite{MR0271527}) associated with KdV.  Specifically, if $u$ is a smooth solution to \eqref{KdV}, then for every
integer $k\geq 0$, there exists polynomials $T^{(k)}$ and $X^{(k)}$ in
$u$ and its derivatives such that
\begin{equation}\label{conslaw}
\partial_t T^{(k)}+ \partial_x X^{(k)}:=0,
\end{equation}
and the following additional properties are satisfied:
\begin{itemize}
\item The polynomial $T^{(k)}$ is irreducible.
\item The \emph{rank}\footnote{Here the rank of a monomial is
  $m+\frac{n}{2}$, where $m$ and $n$ are respectively the degree and
  total number of differentiations of the monomial. Only terms whose rank is an integer occur.} of all monomials
  contained in $T^{(k)}$ is $2+k$.
\item The rank of all monomials
  contained in $X^{(k)}$ is $3+k$.
\item The \emph{dominant}\footnote{Writing a monomial in the form
  $cu^{a_0}u_x^{a_1}\dots(\partial^l_x u)^a_l$, then the dominant term
  is the term with the larger $l$, or the same $l$ but larger $a_l$,
  or with the same $l$ and $a_l$, but with larger $a_{l-1}$, etc.}
  term of $T^{(k)}$ is $(\partial_x^ku)^2$.
\item The polynomial $X^{(k)}$ has two terms with maximal
  \emph{derivative index}\footnote{The derivative index of a monomial
    refers to the number of differentiations in the monomial.}, namely
  $2\partial_x^ku\partial_x^{k+2}u$ and $-(\partial_x^{k+1}u)^2$. 
\end{itemize}

\bibliographystyle{amsplain}
\bibliography{article}

\end{document}